\documentclass[11pt,a4paper]{article}

\usepackage[english]{babel}
\usepackage{graphicx} 
\usepackage{amsthm} 
\usepackage{amsmath}
\usepackage{amssymb}
\usepackage{mathtools}
\usepackage{amsthm} 
\usepackage{wasysym} 
\usepackage{mathrsfs} 
\usepackage{bbm} 
\usepackage{calligra} 
\usepackage{enumerate} 
\usepackage[shortlabels]{enumitem}
\usepackage{xcolor}
\usepackage{multirow} 
\usepackage{subfigure}
\usepackage{nicefrac}
\usepackage{comment}
\usepackage{adjustbox}

\usepackage{hyperref}

\newcommand{\R}{\mathbb{R}}

\newcommand{\T}{\mathbb{T}}

\newcommand{\curveC}{\mathsf C}

\newcommand{\dif}{\,\mathrm{d}}
\DeclareMathOperator{\arcsinh}{arcsinh}

\newtheorem{lemma}{Lemma}[section]
\newtheorem{prop}[lemma]{Proposition}
\newtheorem{thm}[lemma]{Theorem}

\newtheorem{cor}[lemma]{Corollary}
\newtheorem{Rem}[lemma]{Remark}

\title{Dissipative Euler flows originating from \\
circular vortex filaments}
\author{Francisco Gancedo, Antonio Hidalgo-Torn\'e, Francisco Mengual}
\date{\today}

\begin{document}
	
\maketitle

\begin{abstract}
In this paper, we prove the first existence result of weak solutions to the 3D  Euler equation with initial vorticity concentrated in a circle and velocity field in $C([0,T],L^{2^-})$. The energy becomes finite and decreasing for positive times, with vorticity concentrated in a ring that thickens and moves in the direction of the symmetry axis. 
With our approach, there is no need to mollify the initial data or to rescale the time variable.
We overcome the singularity of the initial data by applying convex integration within the appropriate time-weighted space. 
\end{abstract}

\section{Introduction and main result} 

We consider the Cauchy problem for the Euler equation
\begin{subequations}\label{eq:IE}
\begin{align}
\partial_t v + \mathrm{div}(v\otimes v)+\nabla p
&=0,\\
\mathrm{div}v&=0,\\
v|_{t=0}&=v^\circ,
\end{align}
\end{subequations}
posed on $[0,T]\times\R^3$, where $v(t,x)$ is the velocity field and $p(t,x)$ is the scalar pressure. 
The initial velocity $v^\circ$ is given by the Biot-Savart law
\begin{subequations}\label{eq:circular}
\begin{equation}\label{eq:v0}
v^\circ=\nabla\times
(-\Delta)^{-1}\omega^\circ,
\end{equation}
in terms of the vorticity $\omega^\circ$, which we assume to be initially concentrated along a curve, known as a vortex filament.
In this work we focus on the special case
\begin{equation}\label{eq:omega0}
\omega^\circ
=\Gamma\delta_\curveC,
\end{equation}
\end{subequations}
where $\Gamma\neq 0$ is the vorticity strength, and $\curveC$ is an oriented circle with radius $L>0$. Due to the invariances of the Euler equation, we can assume without loss of generality that $\curveC$ is located in the plane $x_3=0$ and centered at the origin $(x_1,x_2)=(0,0)$. 
Our main result reads as follows:

\begin{thm}\label{thm:main} 
There exists $T>0$ such that  \eqref{eq:IE}\eqref{eq:circular} admits infinitely many weak solutions $v\in C([0,T],L^{2^-})$. For any $0<t\leq T$, the energy becomes finite and decreasing with
\begin{equation}\label{eq:thm:energy}
\frac{1}{2}\int_{\R^3}|v(t,x)|^2\dif x
= -\frac{L\Gamma^2}{2}\log c(t) +O(1),
\end{equation}
as $t\to 0$. Furthermore, $v(t)$ is analytic outside the ``turbulence ring''
\begin{equation}\label{eq:thm:tur}
\Omega_{\text{tur}}(t)
=\{x\in\R^3\,:\,\mathrm{dist}(x,h(t)e_3+\curveC)<c(t)\},
\end{equation}
where $c(t)$ and $h(t)$ denote the  thickness and height of the ring respectively
\begin{equation}\label{eq:thm:ch}
c(t)=\sqrt{\nu_{\text{tur}}t},
\quad\quad
h(t)=\frac{\Gamma}{8\pi L}(1-2\log c(t))t,
\end{equation}
and $\nu_{\text{tur}}>0$ represents the ``turbulence viscosity'',
a constant proportional to $|\Gamma|$.
\end{thm}

The proof of Theorem \ref{thm:main} relies on finding a suitable solution to the Euler-Reynolds equation, also known as a ``subsolution'', and later applying the convex integration method \cite{DeLellisSzekelyhidi10}. 
As we will explain in more detail in Sections \ref{sec:Heuristics} and \ref{sec:sketch}, we seek a subsolution whose vorticity concentrates around a filament evolving in time. Heuristically, this evolution should be driven by the binormal flow with a displacement velocity amplified by a factor $\log c$, where $c$ is the radius of the vortex core (see subsection \ref{subsec:overviewliterature} for more details). In the case of the horizontal circle, the binormal direction is the vector $e_3$. If the vorticity remained concentrated in the filament ($c=0$), the circle would ascend to infinity instantly, making it impossible to solve the Euler equation. However, given the singularity of the initial data, it is conceivable that turbulence may be triggered by infinitesimal perturbations, thus expanding the support of the vorticity, slowing down the ascent of the turbulence ring, and decreasing the energy during the process.
In this work we prove that, by choosing $c(t)>0$, it is possible to construct Euler flows satisfying the aforementioned properties. The proper scaling $c(t)\sim\sqrt{t}$ appears naturally as we aim to attain a Reynolds stress comparable to the nonlinear term. 
The constant $\nu_{\text{tur}}$ could be experiment-dependent (see \eqref{eq:thickness}).

Our construction is inspired by prior work \cite{Szekelyhidi11} on the Kelvin-Helmholtz instability for the 2D Euler equation. In this example, the vorticity is initially concentrated in the horizontal axis, that is, $\omega^\circ=\Gamma \delta_{\{x_2=0\}}$. For positive times, the author proved the existence of dissipative Euler flows by devising a subsolution whose vorticity is concentrated inside a turbulence strip $\Omega_{\text{tur}}(t)=\{x\in\R^2\,:\,|x_2|<\nu_{\text{tur}}t\}$. Remarkably, the presence of symmetries simplifies the derivation of the subsolution to a direct computation. Despite our initial data also enjoys certain symmetries, the computation of the subsolution becomes intricate due to various factors.

Firstly, in contrast to the previous works on modeling hydrodynamical instabilities through convex integration, this is the first one in three dimensions. In order to simplify the analysis, we restricted ourselves to the axisymmetric without swirl case. Although this reduces the problem to two dimensions, the corresponding Biot-Savart law involves integrals that cannot be explicitly computed, introducing a series of errors that need to be controlled. To address this, we draw upon previous work \cite{FengSverak15} on viscous vortex rings.

Secondly, our turbulence zone involves three dynamics: vertical displacement, width expansion, and internal rotation. The first example with two dynamics was provided in \cite{CCF21} for the Saffman-Taylor instability in porous media. In this work, the authors proved the existence of mixing solutions by concentrating the vorticity inside a mixing zone $\Omega_{\text{mix}}(t)=\{x\in\R^2\,:\,|x_2-f(t,x_1)|<\nu_{\text{tur}}t\}$ surrounding a dynamical interface $f$. The work  \cite{ForsterSzekelyhidi18} simplified the construction by introducing the idea of concentrating the vorticity on the boundary of the mixing zone. Later, this approach was implemented in \cite{MengualSzekelyhidi23} for the Kelvin-Helmholtz instability with dynamical vortex sheets.

Motivated by the last work \cite{MengualSzekelyhidi23}, our first attempt consisted of trying to concentrate the vorticity at the boundary of the turbulence ring. Unfortunately, this approach did not fit suitably to the requirements of our subsolution. Based on the previous work of the third author on power-law vortices for the 2D Euler equation \cite{Mengualpp}, we had hope in being able to concentrate the vorticity inside the turbulence ring. However, 
in contrast to \cite{Mengualpp}, the errors introduced by the Biot-Savart law prevents the explicit computation of the Reynolds stress. To overcome this difficulty, we made use of the antidivergence operator in \cite{IsettOh16}. This operator has the crucial property of maintaining the support of the turbulence zone, but at the cost of introducing several compatibility conditions. In this work, we ensure these conditions by choosing the vorticity density carefully and introducing a suitable corrector in the pressure.

Thirdly, our initial data is the first one satisfying $v^\circ\notin L_{\text{loc}}^2$ in the context of the Euler equation. Although the h-principle allows for the construction of solutions with finite energy \cite{DeLellisSzekelyhidi10}, it does not immediately extend to our case: the energy blows up as time approaches zero. To address this, we adapt the convex integration method from \cite{DeLellisSzekelyhidi10} for an appropriate time-weighted space.
Although the initial data cannot be reached in $L^2$, we prove strong convergence in all $L^p$ with $1<p<2$, in short, $L^{2^-}$.

\subsection{Open questions and future work}

Our approach seems flexible enough to be applied to more general vortex filament initial data. 
In fact, our goal in an ongoing project is to generalize Theorem \ref{thm:main} and relate it to the dynamics of the binormal flow. 

It is to be expected that the velocities constructed in Theorem \ref{thm:main} can be upgraded inside $\Omega_{\text{tur}}$ to be H\"older continuous below the Onsager critical exponent $\nicefrac{1}{3}$ (see e.g.~\cite{BDSV19}). 
In fact, there is hope for controlling the vorticity in $L^1$ in three dimensions (see e.g.~\cite{JMS21}). This would align with the case of vortex filament data.  
We leave these possible improvements of the regularity for future works.
Instead, our key object of study in the present paper has been the evolution of the turbulence zone and the energy dissipation.

Since the aforementioned convex integration constructions are purely three dimensional, we expect our solutions to break the symmetries of the initial data, namely, they would not be axisymmetric without swirl.
A natural and very interesting question
is whether any solution with the characteristics of Theorem \ref{thm:main} must necessarily break these symmetries.

\subsection{Overview of literature}\label{subsec:overviewliterature}

\subsubsection*{Vortex filaments}
A velocity field is said to be a vortex filament when its vorticity is concentrated in a curve. The study of vorticities concentrated near a curve is a classical problem dating back to the works of Helmhotz \cite{Helmholtz} and Kelvin \cite{Thomson1883}.
An intriguing problem is to determine how the vorticity evolves with time.  The challenge lies in the fact that the velocity over the curve is infinite, making it difficult to formulate the problem.

A common approach is to consider a vorticity tube around the curve and study the evolution of the Euler equation when the width $c$ of the tube vanishes. In \cite{darios1906}, Da Rios derived formally that the velocity over the filament should be, to higher order, $\log c$ in the binormal direction. Therefore, one can study the limit when the width vanishes if simultaneously the time is rescaled as $(\log c)^{-1}$. This leads to the question of whether, with this rescaling, there are Euler solutions whose vorticity is concentrated near a curve that evolves as the binormal flow. That is, the evolution of a curve $\chi(t,s)$, with $s$ the unit speed parametrization, given by the equation
$$
\partial_t \chi(t,s)=\partial_s \chi(t,s)\times \partial^2_s \chi(t,s).
$$
This is known in the literature as the vortex filament conjecture. Some sufficient conditions on Euler solutions to satisfy this conjecture are given in \cite{JerrardSeis2017}. This conjecture has been solved in some particular cases. Part of the difficulty is that the binormal flow itself is a challenging equation that exhibits complex behavior. See the works \cite{jerrardsmets2015,banicavega2022,banicaeceizabarrenanahmodvega2024} and the references therein. In \cite{daviladelpinomussowei2022} this conjecture was proved when the filament is a helix. 
See \cite{CaoWan2023structure}
for a result with helical symmetry on bounded domains. In the axisymmetric case, solutions concentrated close to a curve traveling in its binormal direction were constructed in \cite{fraenkel1970}. In \cite{BenedettoCagliotiMarchioro00}, this result is generalized to more general initial data. See \cite{CaoLaiQinZhan23} and the references therein for more results on the existence, uniqueness, and stability of vortex rings. In \cite{davilaetal2023leapfrogging}, the leapfrogging phenomenon in vortex rings has been rigorously studied. The case of column vortices can be reduced to two dimensions and is equivalent to the point-vortex model. See \cite{MarchioroPulvirenti83, GallaySmets2020} for results in this setting. See also \cite{HHM23} for a recent result on leapfrogging in this setting.

Another approach is to introduce a smoothing mechanism. Considering the Navier-Stokes equation, the global well posedness for small initial vorticity belonging to the Morrey space $\mathcal{M}^{3/2}$, which includes vortex filaments, was demonstrated in \cite{gigamiyakawa1989}. For large data, the local existence and uniqueness under some restrictions of solutions for an arbitrary smooth and closed curve have been demonstrated in \cite{bedrossiangermainharropgriffiths23}. The global existence of solutions without size restriction is known only under symmetry conditions. The translational symmetry case can be reduced to a 2D problem, where there is abundant literature on the subject. See \cite{giga1988} for an existence and \cite{gallaghergallaylions2005} for a uniqueness result. For axisymmetric fluids without swirl, the global existence of axisymmetric solutions was shown in \cite{FengSverak15}, and the uniqueness in the symmetry class in \cite{GallaySverak19}. For helical fluids, the global existence of solutions has been proved recently \cite{GH-T2023}. The drawback of considering the Navier-Stokes equation is that the vorticity support instantaneously becomes the whole space. However, it is expected that the vorticity remains concentrated mainly near a curve for times $t\ll \nu^{-1}$. Regarding this approach, when the initial vorticity is supported in a circle, it has been shown in \cite{GallaySverak19} that the vorticity of the unique axisymmetric solution remains close to a curve moving in its binormal direction. This result has been sharpened in \cite{gallaysverak2023} for vorticities sufficiently large compared to viscosity. For a small initial vorticity supported on a closed curve, it was proven recently in \cite{FontelosVega23} that there is a solution with the vorticity remaining close to the evolution of the curve in its binormal direction.

In this paper, we consider a new approach that allows weak solutions to the Euler equation with initial vorticity concentrated in a curve, without time rescaling or regularization of the initial data. As we have mentioned before, convex integration can help us to overcome the singularity of the initial data. 

This approach is natural if one takes into account that high-energy filaments are expected to be unstable (see e.g. \cite{WidnallSullivan73} and \cite[Section 5]{vandyke1982album}).
In this regard, it is worth noting that the nonuniqueness of Leray solutions to the forced Navier-Stokes equation has been proved recently for a specific unstable vortex ring \cite{ABC23}.

\subsubsection*{Convex integration}
In the seminal work \cite{DeLellisSzekelyhidi09}, non-trivial Euler flows are constructed by adapting the convex integration method, which traces back to works in differential geometry \cite{Nash54} and calculus of variations \cite{MullerSverak03}, in hydrodynamics. Given that the solutions in \cite{DeLellisSzekelyhidi09} satisfy $v^\circ=0$, their energy must necessarily increase. In their second paper \cite{DeLellisSzekelyhidi10}, upon which the present work is based, they initiated the investigation on ``admissible'' Euler flows, those characterized by non-increasing energy. After a decade of refinements to this method, it allowed to construct Euler velocities with H\"older regularity strictly less than $\nicefrac{1}{3}$ exhibiting nonuniqueness \cite{Isett18} and dissipating the kinetic energy \cite{BDSV19}. This milestone thereby resolves the dissipative part of the 3D Onsager conjecture. More recently, the $L^3$-based Onsager conjecture was proved in \cite{GNKpp}. The conservative part of the Onsager conjecture had been proven previously in \cite{Eyink94,CET94}. See also \cite{CCFS08} for critical regularity.
Remarkably, this method allowed for proving nonuniqueness of solutions to the Navier-Stokes equation \cite{BuckmasterVicol19} and sharpness of the Ladyzhenskaya-Prodi-Serrin criteria \cite{CheskidovLuo22}.
In the context of turbulent transport,
it was recently shown in \cite{BSWpp} anomalous dissipation
of scalar fields advected by typical Euler flows below the Onsager threshold. 

In the aforementioned results, the emphasis is placed on upgrading the global regularity of the solutions, rather than addressing the Cauchy problem for specific data. 
According to the weak-strong uniqueness principle, initial data exhibiting nonuniqueness of admissible solutions, known as ``wild'' data, cannot be smooth.  In \cite{DaneriSzekelyhidi21} it is proved that wild data are dense in the space of divergence-free vector fields. The underlying idea is that a small but rough perturbation can be added to any (smooth) data, triggering turbulence. Despite this genericity, determining whether a given initial data is wild, that is, constructing admissible solutions emanating from it, turns out to be quite challenging. Recently, the third author \cite{Mengualpp} proved the existence of wild data in two dimensions satisfying $\omega^\circ\in L^1\cap L^p$ for any given $p<\infty$.
This showcases both sharpness of the weak-strong uniqueness principle and Yudovich's proof of uniqueness \cite{YUDOVICH63}. The construction relies on devising a subsolution in the form of a self-similar vortex, with $\Omega_{\text{tur}}(t)=\{x\in\R^2\,:\,|x|\leq (\nu_{\text{tur}}t)^{\nicefrac{p}{2}}\}$ and $\nu_{\text{tur}}\sim |\Gamma|p^{-2}$. As we will in Section \ref{sec:sketch}, the cross section of our vortex ring shares similar properties (see \eqref{compcond2}) to that of \cite{Mengualpp} when $p\to 1$.

The difficulty in determining if a given data is wild arises from the fact that these are often associated with some form of hydrodynamical instability.
As we mentioned at the beginning of the introduction, \cite{Szekelyhidi11} provided the first example of wild data in the context of the Kelvin-Helmholtz instability  (see also \cite{MengualSzekelyhidi23}). This initiated a promising program of modeling hydrodynamical instabilities through convex integration. Apart from the Kelvin-Helmholtz instability, this approach has been successfully applied to the Rayleigh-Taylor instability \cite{GKS21,GebhardKolumban22a,GebhardKolumban22b,GHKpp}, and the Saffman-Taylor instability 
(see e.g.~\cite{Szekelyhidi12,CCF21,ForsterSzekelyhidi18,NoisetteSzekelyhidi21,CFG23} and the references therein). 

In these works, the fluid is initially smooth outside an interface, where the hydrodynamical instability is concentrated. The general approach consists of opening a turbulent zone around a dynamical interface. Outside the turbulent zone, the corresponding Reynolds stress must vanish, meaning it must be a solution of the original system. Inside the turbulent zone, whenever the h-principle applies, it is sufficient to construct a subsolution. Despite the degree of freedom introduced by the Reynolds stress, the physical properties we seek in our solutions restrict the evolution of the turbulent zone. As an illustrative example, \cite{CFM22} showed that the thickness of the mixing zone is controlled by the inequality $c(t,\alpha)\leq t(1-\sigma(t,\alpha))$, where $\sigma$ is the Saffman-Taylor function, which takes values in $[-1,1]$. This constraint prevents the fluids from mixing at asymptotically stable points ($\sigma=1$), thus forcing the mixing zone to be localized around the instability.

Although nonuniqueness appears inherent in these phenomena, there is hope for restoring determinism at the average level by selecting the  Reynolds stress (equivalently the subsolution) as a function of the average flow. 
This challenge is known in the literature as the ``closure problem''. 
For the Kelvin-Helmholtz instability, \cite{Szekelyhidi11} and \cite{MengualSzekelyhidi23} studied a selection criterion based on maximal energy dissipation. For the Rayleigh-Taylor instability, \cite{GHKpp}
 proposed a selection criterion based on the least action principle. Regarding the Saffman-Taylor instability,  \cite{Otto99} and \cite{Szekelyhidi12} obtained a selection criterion based on a Lagrangian and Eulerian relaxation approach, respectively (see also \cite{CFM19,Mengual22,HitruhinLindberg21}). For the flat interface, this approach yields a unique entropy solution as the candidate for the macroscopic solution.
Recently, in \cite{CFG23} these entropy solutions for dynamical interfaces have been constructed.

\subsubsection{General notation}
In this subsection we clarify some notation.
\begin{itemize}
    \item  Given two comparable quantities $a,b$, we denote $a\lesssim b$ if
there exists a constant $C>0$ such that $a\leq Cb$, and also $a\sim b$ if additionally $b\lesssim a$.
    \item We will use $x=(x_1,x_2,x_3)$ as cartesian coordinates, $(r,\theta,z)$ as cylindrical coordinates, and $(\rho,\alpha)$ as the polar coordinates of the cross section of the turbulence ring.
    \item In axisymmetric coordinates, we will identify
    $$
    \mathbb{H}:=\{(r,z)\in\R^2\,:\, r>0\},
    $$
    with the complex half-plane, and we will write compactly
    $$
    \zeta=r+iz.
    $$
    Moreover, we will denote $\zeta^*=r-iz$ and $\zeta^\perp=i\zeta=-z+ir$.
    \item We will denote $\Omega_{\text{tur}}(t)$ by the turbulence ring at time $t$, and $\Omega_{\text{tur}}=\{(t,x)\,:\,x\in\Omega_{\text{tur}}(t)\}$ by the space-time turbulence zone.
\end{itemize}

\subsection{Organization of the paper}

We start by providing some heuristics and outlining the proof of Theorem \ref{thm:main} in Sections \ref{sec:Heuristics} and \ref{sec:sketch} respectively.
In Section \ref{sec:Axisymmetric} we derive the Euler-Reynolds equation for axisymmetric without swirl subsolutions. In Section \ref{sec:subsolution} we construct the subsolution, while the energy is computed in Section \ref{sec:energy}. In the Appendix we include the axisymmetric without swirl Biot-Savart law and the antidivergence operator, respectively.

\section{Heuristics}\label{sec:Heuristics}

Theorem \ref{thm:main} gives the first example of weak solutions to the Euler equation with vortex filament data. For this reason, we would like to provide a heuristic explanation of its existence.

Let us consider the Cauchy problem for the 3D Navier-Stokes equation
\begin{subequations}\label{eq:NS}
\begin{align}
\partial_t v_\nu + \mathrm{div}(v_\nu\otimes v_\nu)+\nabla p_\nu
&=\nu\Delta v_\nu,\\
\mathrm{div}v_\nu&=0,\\
v_\nu|_{t=0}&=v^\circ,
\end{align}
\end{subequations}
where $\nu>0$ is the kinematic viscosity.

In \cite{GallaySverak19}, it is proved that, within the
axisymmetric without swirl class, there is a unique Navier-Stokes flow $v_\nu$ originating from the circular vortex filament \eqref{eq:circular}, which becomes smooth for positive times. 
This solution behaves like a vortex ring of thickness proportional to
$$
c_\nu(t)=\sqrt{\nu t},
\quad\quad
\nu=\frac{|\Gamma|}{\text{Re}},
$$
which is located, according to the Kelvin-Saffman formula, in the $x_3$-axis with height 
$$
h_\nu(t)=\frac{\Gamma}{8\pi L}(C-2\log c_\nu(t))t,
$$
where $\text{Re}>0$ is the (circulation) Reynolds number, and 
$C$ is a dimensionless constant that depends on the distribution of vorticity inside the
cross section of the ring. 
Observe that, as the Reynolds number increases, this vortex ring ascends more and more rapidly, while vorticity dissipation decreases. Formally, in the inviscid limit $\nu\to 0$, 
which corresponds to $\text{Re}\to\infty$ when $\Gamma$ is fixed,
one would observe the circle ascending towards to infinity instantaneously. For this reason, one might tend to think that there cannot exists Euler flows originating from the
circular vortex filament \eqref{eq:circular}, thus contradicting Theorem \ref{thm:main}. Our goal in this section is to explain why this is not necessarily the case.

At this point, we would like to caution the reader that, although it is tempting to draw parallels between the $c$ and $h$ in our Theorem \ref{thm:main} with $c_\nu$ and $h_\nu$ above, our scaling arises from the (local) self-similarity of the initial data and turbulent dissipation, rather than from viscous dissipation.

\subsection{Spontaneous stochasticity and anomalous dissipation}

As already mentioned in \cite{GallaySverak19}, a natural question is whether uniqueness remains true among all reasonable solutions that approach the initial vortex filament in a suitable sense.
Both from an experimental and numerical point of view, it is indeed more realistic to consider potential deviations of the initial data.
In this regard, we consider 
\begin{equation}\label{eq:omegaper}
v_\eta^\circ=\nabla\times
(-\Delta)^{-1}\omega^\circ_\eta,
\quad\quad
\omega^\circ_\eta
=\Gamma\delta_{(\curveC+\eta)},
\end{equation}
for some perturbations $\eta$ possibly breaking the symmetries of the circular filament \eqref{eq:circular}. 
If preferred, one could consider smoother data by mollifying \eqref{eq:omegaper} or thickening the filament $\curveC+\eta$.
As we aim to examine small perturbations, we consider a family $\varepsilon_\nu P$, for some fixed class of perturbations $P$, and some $\varepsilon_\nu>0$ that vanishes as $\nu\to 0$.

Whether what we observe later  resembles the ideal case $v_\nu$ or not will depend on the stability of the system.
It was believed, starting with the works of Kelvin \cite{Thomson1883}, that vortex rings in an ideal fluid were ``indestructible''. However, subsequent studies (see e.g.~\cite{WidnallSullivan73}) indicate that vortex ring becomes more unstable as the ratio $c/L$ decreases, being $c$ the radius of the vortex core, and $L$ the radius of the ring. If $L$ is fixed, the limit $c\to 0$ corresponds to the case of vortex filaments.

In this sense, it is conceivable that these deviations could writhe the viscous ring, slowing down the ascent and disintegrating the vorticity into a thicker region of turbulence during the process. 
In the event that this turbulent region continues to be significant for a large class of perturbation $\eta\in\varepsilon_\nu P$ in the inviscid limit ($\nu\to 0$), one might expect the Euler equation \eqref{eq:IE} for the singular initial data \eqref{eq:circular} to lose determinism ($\varepsilon_\nu\to 0$), a phenomenon known in the literature as ``spontaneous stochasticity'' (see e.g.~\cite{EyinkDrivas15,TBM20}). This would align with the nonuniqueness and the turbulence ring  described in our Theorem \ref{thm:main}.

In \cite{bedrossiangermainharropgriffiths23}, it is proved that, in a suitable class, there is a local in time unique Navier-Stokes flow $v_{\eta,\nu}$ originating from \eqref{eq:omegaper}, whose energy becomes finite for positive times.
These solutions can be extended globally in time by picking one of the (potentially non-unique) Leray solutions. It is well-known that these solutions satisfy the energy inequality
$$
\frac{1}{2}\int_{\R^3}
|v_{\eta,\nu}(t,x)|^2\dif x
\leq
\frac{1}{2}\int_{\R^3}
|v_{\eta,\nu}(t_\nu,x)|^2\dif x
-\nu\int_{t_\nu}^t
\int_{\R^3}
|Dv_{\eta,\nu}(s,x)|^2\dif x\dif s,
$$
for any $t_\nu>0$. 
According to the zeroth law of turbulence, it is predicted that the last term will remain uniformly negative in the inviscid limit due to the loss of regularity of turbulent flows.
This phenomenon is also known in the literature as ``anomalous dissipation'' (see e.g.~\cite{BrueCamillo23}).
We believe that, in our particular case, this phenomenon could compensate for the infinite initial energy and eventually lead to the production of dissipative Euler flows, as described in our Theorem \ref{thm:main}.
\bigskip

\begin{figure}[ht!]
	\centering
 \begin{subfigure}{}
\centering
	\adjincludegraphics[height=4.5cm,trim={{.30\width} {.25\width} {.27\width} {.15\width}},clip]{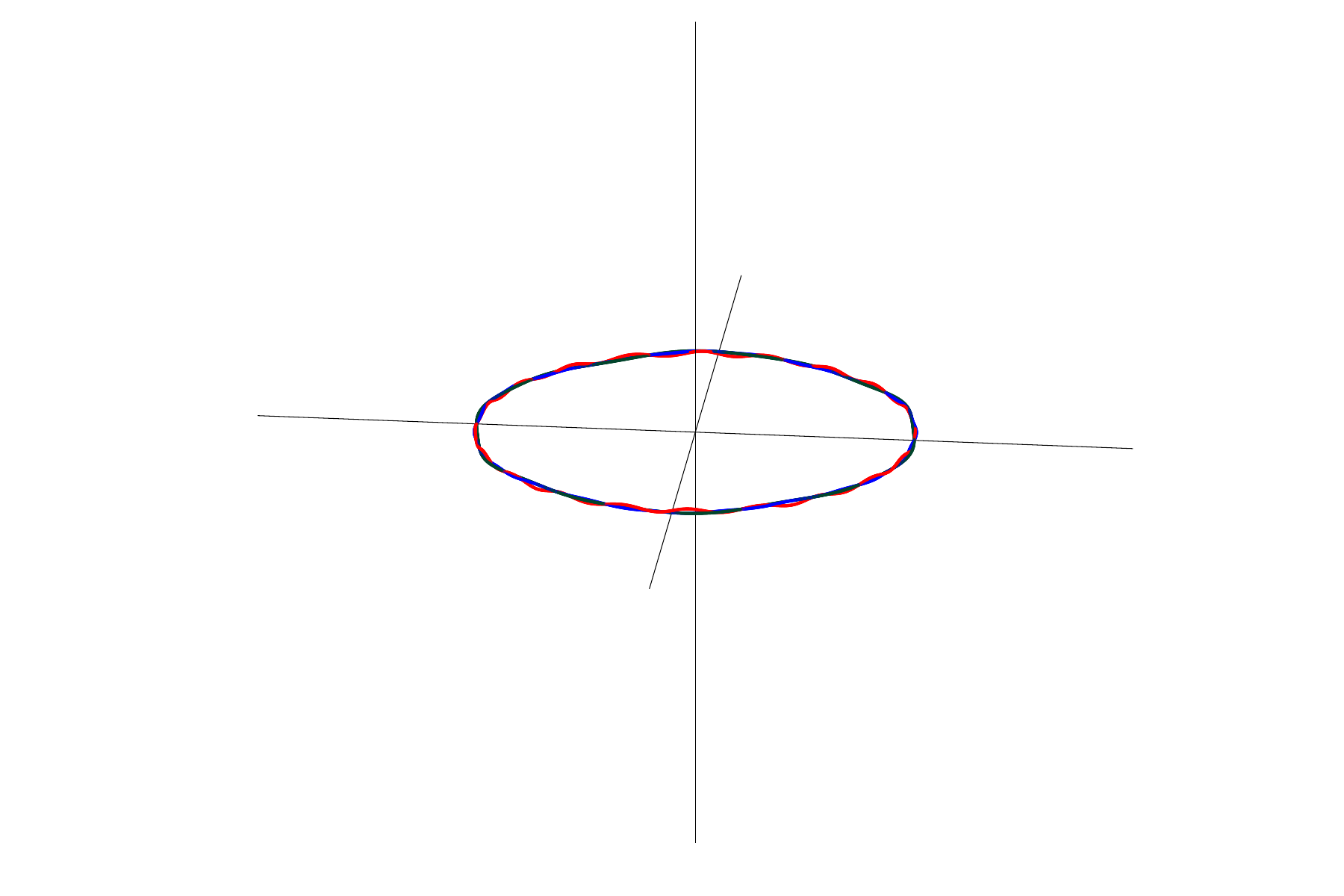}
\end{subfigure}
\begin{subfigure}{}
\centering
	\adjincludegraphics[height=4.5cm,trim={{.30\width} {.25\width} {.27\width} {.15\width}},clip]{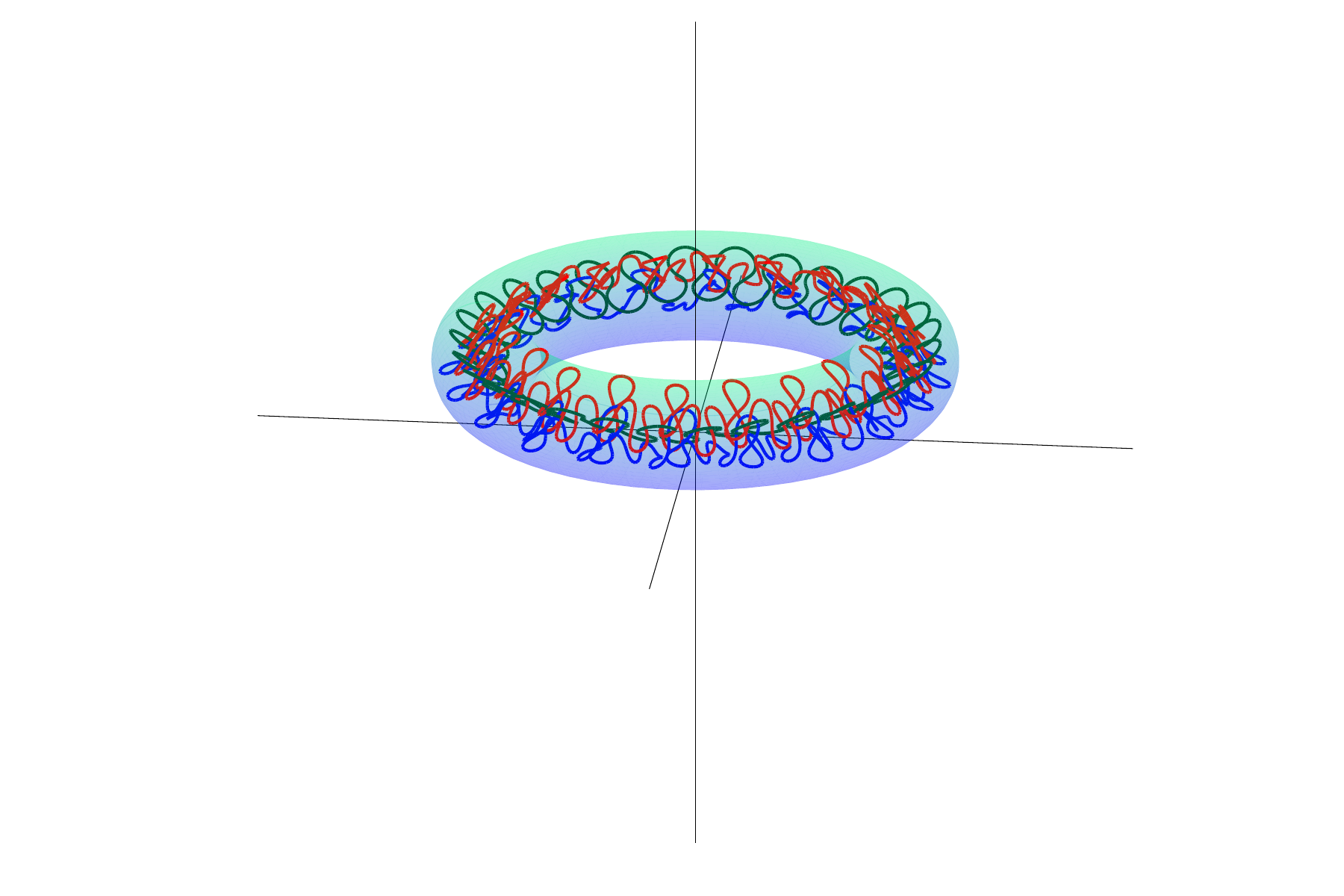}
\end{subfigure}
\bigskip
	\caption{Cartoon of three different vortex filaments, originating from perturbations of the ideal circle, inside the turbulence ring. Left: The initial perturbations $\curveC+\eta$. Right: The vortex filaments at a positive time $t$, included in a ring representing the turbulence zone $\Omega_{\text{tur}}(t)$.}
 \label{fig:turbulencering}
\end{figure}

\subsection{The Reynolds stress}

While nonuniqueness appears inherent to the phenomenon of spontaneous stochasticity, there is still hope to restore determinism by averaging these turbulent solutions. With this aim, we consider some probability measure $\mu$ defined on $P$ (that extends trivially to $\varepsilon_\nu P$), with which we define the ``averaged solution''
\begin{equation}\label{eq:average}
\bar{v}_{\nu}
=\int_{\varepsilon_\nu P}
v_{\eta,\nu}\dif\mu.
\end{equation}
Given that averages do not commute with nonlinearities, this averaged solution satisfies the Navier-Stokes-Reynolds equation
\begin{subequations}\label{eq:NSR:0}
	\begin{align}
	\partial_t\bar{v}_\nu + \mathrm{div}(\bar{v}_\nu\otimes \bar{v}_\nu +R_\nu)+\nabla\bar{p}_\nu
	&=\nu\Delta\bar{v}_\nu,\\
	\mathrm{div}\bar{v}_\nu&=0,\\
	\bar{v}_\nu|_{t=0}&=v^\circ,
	\end{align}
\end{subequations}
where $\bar{p}_\nu$ denotes the averaged pressure, and $R_\nu$ is the ``Reynolds stress''
$$
R_\nu
=\overline{v_\nu\otimes v_\nu}-\bar{v}_\nu\otimes \bar{v}_\nu
=\overline{(v_\nu-\bar{v}_\nu)}\otimes\overline{(v_\nu-\bar{v}_\nu)}.$$
Notice that $R_\nu$ is a positive definite symmetric tensor.
Since $\mathrm{tr}R$ can be absorbed in the pressure, the deviation from being a true Navier-Stokes flow is captured by the traceless Reynolds stress
$$
\mathring{R}_\nu=R_\nu-\frac{1}{3}(\mathrm{tr}R_\nu)I_3.
$$
In turn, this term contributes to the energy density by
$$
e_{\text{sub}}(\bar{v}_\nu,\mathring{R}_\nu)
=\frac{1}{2}|\bar{v}_\nu|^2
+\frac{3}{2}\lambda_{\max}(\mathring{R}_\nu),
$$
where $\lambda_{\max}$ is the largest eigenvalue, which is nonnegative for traceless symmetric matrices. Notice that $e_{\text{sub}}$ coincides with the usual energy density if and only if $\mathring{R}_\nu=0$. 
As discovered in \cite{DeLellisSzekelyhidi10},
the energy density $e_{\text{sub}}$ characterizes the relaxation of the Euler equation, in particular it is convex within the space of subsolutions. 
By applying the Jensen inequality, we deduce that the energy of the averaged solution remains below the averaged energy
$$
\int_{\R^3}e_{\text{sub}}(\bar{v}_\nu,\mathring{R}_\nu)\dif x
\leq\int_{\varepsilon_\nu P}\frac{1}{2}\int_{\R^3}
|v_{\eta,\nu}(t,x)|^2\dif x\dif\mu.
$$
In the inviscid limit, we obtain formally a solution to the Euler-Reynolds equation
\begin{subequations}\label{eq:IER}
	\begin{align}
	\partial_t\bar{v} + \mathrm{div}(\bar{v}\otimes \bar{v})+\nabla\bar{p}
	&=-\mathrm{div}R,\label{eq:IER:1}\\
	\mathrm{div}\bar{v}&=0,\label{eq:IER:2}\\
	\bar{v}|_{t=0}&=v^\circ.
	\end{align}
\end{subequations}
Notice that the nonlinear term may not necessarily converge. However, since it is expected that the averages cancel out highly oscillating terms, let us assume, for simplicity, that this is not the case (otherwise, this contribution would be added to $R$). 
Finally, the Fatou lemma allows to bound the energy by
\begin{align*}
\int_{\R^3}e_{\text{sub}}(\bar{v},\mathring{R})\dif x
&\leq\liminf_{\nu\to 0}\int_{\R^3}e_{\text{sub}}(\bar{v}_\nu,\mathring{R}_\nu)\dif x.
\end{align*}

We have seen that the existence of a subsolution (a solution to the Euler-Reynolds equation) with decreasing energy is a necessary condition for the occurrence of spontaneous stochasticity and anomalous dissipation. As mentioned in the previous section, it is natural to assume that the vorticity of this subsolution accumulates in a turbulence ring that expands around a dynamical filament and makes the energy finite for positive times. Figure \ref{fig:turbulencering} shows a cartoon illustrating these ideas.

By virtue of the h-principle for the Euler equation \cite{DeLellisSzekelyhidi10}, the existence of weak solutions  as described in Theorem \ref{thm:main} is reduced to find a suitable subsolution.
We remark that it remains an open question whether these solutions arise from an inviscid limit as described before.

Notice that the Reynolds stress introduces several degrees of freedom in the system, making the problem overdetermined. Despite this freedom, we will see in the next section how the physical properties we seek end up roughly determining the evolution of the turbulent zone and the energy dissipation.

\section{Sketch of the proof}\label{sec:sketch}

In this section we outline the main steps for the proof of Theorem \ref{thm:main}.
The precise computations and technical details will be rigorously addressed in the following sections.

Our first task consists of constructing a subsolution $(\bar{v},\bar{p},R)$, that is, a solution to the Euler-Reynolds equation \eqref{eq:IER}.
If we interpret the subsolution as the averaged solution \eqref{eq:average}, it makes sense for it to preserve the symmetries of the initial data. In our case, the initial data is axisymmetric without swirl. Specifically, $\omega^\circ=\Gamma\delta_{\curveC}$
is the vector-valued measure satisfying 
$$
\langle \omega^\circ,\varphi\rangle
=\Gamma\int_{\curveC}
\varphi\cdot\dif l
=\Gamma L\int_0^{2\pi}\varphi(Le_r)\cdot e_\theta\dif\theta,
$$
for any test function $\varphi\in C(\R^3)$,
where
$$e_r=(\cos\theta,\sin\theta,0),
\quad
e_\theta = (-\sin\theta,\cos\theta,0),
\quad
e_z=(0,0,1),$$
is the standard orthonormal basis in cylindrical coordinates $(r,\theta,z)$.
Hence, we assume that the vorticity of the subsolution 
is of the form
\begin{equation}\label{ansatz:omegaaxi}
\bar{\omega}(t,x)
=\bar{\omega}_\theta(t,r,z)e_\theta,
\end{equation}
for some vorticity flux $\bar{\omega}_\theta$, which we assume to be bounded for positive times, to be determined. Notice that $\mathrm{div}\bar{\omega}=0$.
It is well-known (see Appendix \ref{sec:ABS}) that the corresponding velocity field
is of the form 
\begin{equation}\label{ansatz:v}
\bar{v}(t,x)
=-\frac{1}{4\pi}\int_{\R^3}\frac{(x-x')\times\bar{\omega}(t,x')}{|x-x'|^3}\dif x'
=\bar{v}_r(t,r,z)e_r+\bar{v}_z(t,r,z)e_z.
\end{equation}
As we aim to attain $R(t,x)\sim\bar{v}\otimes\bar{v}$, we assume that the Reynolds stress is of the form
\begin{equation}\label{ansatz:R}
R(t,x)
=R_{rr}(t,r,z)(e_r\otimes e_r)
+R_{rz}(t,r,z)(e_r\otimes e_z+e_z\otimes e_r) + R_{zz}(t,r,z)(e_z\otimes e_z).
\end{equation}
Similarly, we assume that 
\begin{equation}\label{ansatz:p}
\bar{p}(t,x)=\bar{p}(t,r,z).
\end{equation}
Notice that all the coefficients in \eqref{ansatz:omegaaxi}-\eqref{ansatz:p} do not depend on the azimuth $\theta$. This makes possible to write the Euler-Reynolds equation \eqref{eq:IER} in the (complex) half-plane
\begin{equation}\label{eq:H}
\mathbb{H}:=\{(r,z)\in\R^2\,:\, r>0\},
\end{equation}
by identifying (with slightly abuse of notation)
\begin{equation}\label{eq:vRaxi}
\bar{v}=\left[\begin{array}{c}
\bar{v}_{r} \\
\bar{v}_{z}
\end{array}\right],
\quad\quad
R=\left[\begin{array}{cc}
R_{rr} & R_{rz} \\
R_{rz} & R_{zz}
\end{array}\right].
\end{equation}
Under the axisymmetric without swirl assumption \eqref{ansatz:omegaaxi}-\eqref{ansatz:p}, we prove in Section \ref{sec:Axisymmetric} that the momentum equation \eqref{eq:IER:1} is rewritten for \eqref{eq:vRaxi} as
\begin{equation}\label{eq:IERaxi:q:0}
\partial_t \bar{v}-\bar{\omega}_\theta \bar{v}^\perp+\nabla \bar{q}
=-\frac{1}{r}\mathrm{div}(rR),
\end{equation}
on $[0,T]\times\mathbb{H}$, where 
$$
\bar{q}=\bar{p}+\frac{1}{2}|\bar{v}|^2,
$$
is the Bernoulli pressure. We remark that, by the Biot-Savart law \eqref{ansatz:v}, the velocity field is divergence-free \eqref{eq:IER:2} and satisfies the boundary condition
$\bar{v}_r=0$ on $\partial\mathbb{H}$.

As we were arguing in the previous sections, it is natural to assume that the vorticity is supported in a ring that moves in the vertical direction, and expands outwards. Thus, the cross section 
of the vortex ring in the complex half-plane \eqref{eq:H} is the shifted ball
\begin{equation}\label{eq:shiftedball}
L+ih(t)+B_{c(t)},
\end{equation}
where $h(t)$ and $c(t)$ denote the height 
and thickness of the ring respectively (see Figure \ref{Fig:crosssection}).
Additionally, due to the vorticity, this ring has an internal rotation with angular speed $a(t,\rho)$. 
Hence, we parameterize the section of the vortex ring by
\begin{equation}\label{ansatz:gamma}
\gamma(t,\rho,\alpha)
=L+ih(t)+c(t)\rho e^{i(\alpha+a(t,\rho))},
\end{equation}
for $(\rho,\alpha)\in[0,1]\times\T$.
Notice that $\gamma(t)$ defines a diffeomorphism with Jacobian
$$
\det(\nabla\gamma(t))
=\rho c(t)^2.
$$
Thus, we take the vorticity flux vanishing outside $\gamma$, while inside it is given by
\begin{equation}\label{ansatz:omegatheta}
\bar{\omega}_\theta(t,\gamma)
=\frac{\varpi(\rho)}{c(t)^2},
\end{equation}
where $\gamma=\gamma(t,\rho,\alpha)$,
for some profile $\varpi$ to be determined. 
The denominator in \eqref{ansatz:omegatheta} is a normalization factor necessary to achieve the initial condition \eqref{eq:circular}, that is,
$$
\bar{\omega}(t)\overset{*}{\rightharpoonup}\Gamma\delta_{\curveC},
$$
in the space of measures as $t\to 0$.
Specifically, 
$\varpi$ is normalized in such a way that the vorticity flux is conserved
\begin{equation}\label{compcond1}
\int_{\mathbb{H}}\bar{\omega}_\theta\dif\zeta=
2\pi\int_0^1\varpi\rho\dif\rho=\Gamma.
\end{equation}
We remark that we could choose $\varpi$ to depend on $t$ and $\alpha$, but since this additional freedom is not needed for our purposes, we let $\varpi$ depend only on the internal radius $\rho$ for simplicity.

\begin{figure}[h!]
	\centering
 \adjincludegraphics[height=6.0cm,trim={{.35\width} {.11\width} {.35\width} {.12\width}},clip]{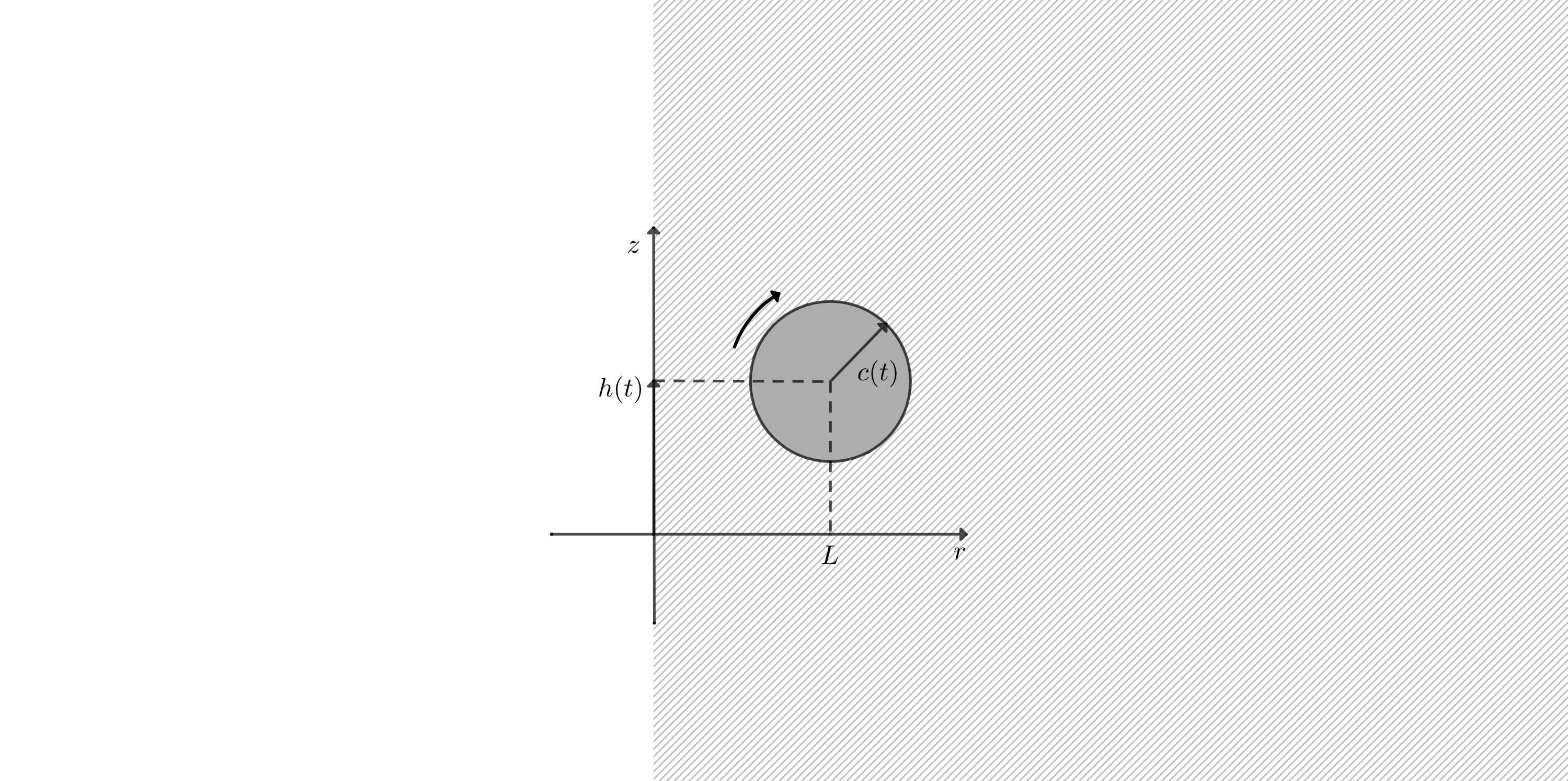}
\caption{The vorticity flux $\bar{\omega}_\theta(t)$ is supported in the shifted ball \eqref{eq:shiftedball}. The dashed area corresponds to the complex half-plane $\mathbb{H}$.}
\label{Fig:crosssection}
\end{figure}

Once the vorticity is prescribed, the velocity is determined by the Biot-Savart law \eqref{ansatz:v}. 
In Section \ref{sec:velocity} we decompose the velocity inside the vortex ring into 
$$
\bar{v}=\bar{v}_{\uparrow}+\bar{v}_{\circlearrowright},
$$
where
\begin{equation}\label{eq:vupcircle}
\bar{v}_{\uparrow}=-\frac{i\Gamma}{4\pi L}\log c+O(1),
\quad\quad 
\bar{v}_{\circlearrowright}=-\frac{i\rho\Gamma_\rho}{2\pi c} e^{i(\alpha+a)}+O(1),
\end{equation}
as $t\to 0$. 
The terms $\bar{v}_{\uparrow}$ and $\bar{v}_{\circlearrowright}$ correspond to the vertical displacement and internal rotation of the ring respectively. For $\Gamma>0$ (respectively $\Gamma<0$) the ring moves upward (downward) along the symmetry axis. In each layer $\rho$ of the ring, the particles rotate (counter) clockwise with speed proportional to

\begin{equation}\label{eq:Gamma_rhodefinition}
\Gamma_\rho
\coloneqq 2\pi\int_0^{1}\varpi(\rho\lambda)
\lambda\dif\lambda.    
\end{equation}

Once the velocity is known, we determine the pressure by means of the Bernoulli law.
More precisely, in Section \ref{sec:pressure} we split the Bernoulli pressure into $\bar{q}=\bar{q}_0+\bar{q}_1$. The main pressure 
$\bar{q}_0$ is taken suitably in such a way that $(\bar{v},\bar{q})$ becomes a (true) solution to the Euler equation outside the turbulence zone $\Omega_{\text{tur}}$, namely
\begin{equation}\label{eq:vq0}
\partial_t\bar{v}+\nabla\bar{q}_0
=\bar{\omega}_\theta\partial_t\gamma^\perp.
\end{equation}
The second pressure $\bar{q}_1$ is a smooth corrector, which is supported inside $\Omega_{\text{tur}}$, to be determined.

Once the velocity and the pressure are known, it remains to determine the Reynolds stress. 
In light of \eqref{eq:vq0}, the momentum equation \eqref{eq:IERaxi:q:0} reads as
\begin{equation}\label{eq:divR}
-\frac{1}{r}\mathrm{div}(rR)
= \bar{\omega}_\theta(\partial_t\gamma - \bar{v})^\perp
+\nabla \bar{q}_1.
\end{equation}
In other words, the Reynolds stress measures the discrepancy between the evolution of the vortex ring and the averaged velocity, which we aim to minimize.
By applying the expressions \eqref{ansatz:gamma} and \eqref{eq:vupcircle}, this discrepancy satisfies the asymptotic
\begin{equation}\label{eq:gamma-v}
\partial_t\gamma-\bar{v}
=
\dot{c}\rho e^{i(\alpha+a)}
+i\left(\dot{h}+\frac{\Gamma}{4\pi L}\log c\right)
+\left(c\dot{a}+\frac{\Gamma_\rho}{2\pi c}\right)i\rho e^{i(\alpha+a)} + O(1),
\end{equation}
as $t\to 0$.
The first term arises from opening the turbulence ring. The only possible way to eliminate it is by taking $c(t)=0$, but then, as we mentioned at the beginning of the introduction, the Euler equation cannot be solved. Once we assume $c(t)>0$ for positive times, we can cancel the second term in \eqref{eq:gamma-v} by choosing the height of the ring satisfying
\begin{equation}\label{eq:height}
\dot{h}(t)
=-\frac{\Gamma}{4\pi L}\log c(t).
\end{equation}
Similarly, we can cancel the third term in \eqref{eq:gamma-v} by choosing the angle speed satisfying
\begin{equation}\label{eq:angle}
\dot{a}(t,\rho)=-\frac{\Gamma_\rho}{2\pi c(t)^2}.
\end{equation}
The equations \eqref{eq:height} and \eqref{eq:angle} determine, for short times, the evolution and shape of the turbulence ring  \eqref{ansatz:gamma} in terms of $c$, which still needs to be determined. Thus, \eqref{eq:gamma-v} reduces to
\begin{equation}\label{eq:gamma-v:1}
\partial_t\gamma-\bar{v}
=
\dot{c}\rho e^{i(\alpha+a)}+ O(1).
\end{equation}
Therefore, the momentum equation \eqref{eq:divR} reads as
\begin{equation}\label{eq:divR:1}
-\frac{1}{r}\mathrm{div}(rR)
=\frac{\varpi}{c^2}\left(\dot{c}\rho e^{i(\alpha+a)}+O(1)\right)+\nabla\bar{q}_1.
\end{equation}
As we aim to concentrate the turbulence inside the support of the vorticity, 
we need to invert the divergence operator in such a way that 
\begin{equation}\label{eq:suppR}
\mathrm{supp}(R(t))\subset L+ih(t) + B_{c(t)}.
\end{equation}
Given that $R$ must be a symmetric matrix, it turns out that the Bogovskii operator, despite being a seemingly natural choice, is not suitable for this particular application. As we will explain in Appendix \ref{section:antidivergence}, by employing a Fourier expansion already present in the work  \cite{IsettOh16}, it is possible to invert the divergence satisfying \eqref{eq:suppR} provided that two compatibility conditions are satisfied. 
We show in Section \ref{sec:Reynolds} that
these conditions can be ensured by taking the corrected pressure $\bar{q}_1$ properly, and the vorticity profile satisfying
\begin{equation}\label{compcond2}
\int_0^1\varpi(\rho)\rho^3\dif\rho=0.
\end{equation}
It is straightforward to check that this condition is equivalent to the one in \cite[Lemma 4.1]{Mengualpp}.

As this antidivergence operator gains a factor $c$ from the thickness of the support \eqref{eq:suppR}, it turns out that the corresponding Reynolds stress \eqref{eq:divR:1} behaves for short times like
$$
R\sim \frac{\dot{c}}{c}.
$$
Since $\bar{v}\sim\frac{1}{c}$ and we expect 
$$
R\sim\bar{v}\otimes\bar{v}\sim\frac{1}{c^2},
$$ we finally obtain the scaling compatibility condition  $c\dot{c}\sim 1$,
or equivalently,
\begin{equation}\label{eq:thickness}
c(t)=\sqrt{\nu_{\text{tur}}t},
\quad\quad
\nu_{\text{tur}}=\frac{|\Gamma|}{\text{Re}_{\text{tur}}},
\end{equation}
where $\nu_{\text{tur}}$ is the turbulence viscosity, and $\text{Re}_{\text{tur}}$ is the turbulence Reynolds number (see e.g.~\cite{Pope00}).
While $|\Gamma|$ arises from dimensional analysis,  we believe that $\text{Re}_{\text{tur}}$ could be experiment-dependent. For instance, if our experiment were able to replicate the ideal conditions \eqref{eq:IE}\eqref{eq:circular} without any external noise, then we would have $\nu_{\text{tur}}=0$. However, as we argued in the previous sections, it is conceivable that $\nu_{\text{tur}}>0$ depending on the perturbation $\varepsilon_\nu P$.

Once $c$ is determined by \eqref{eq:thickness}, the height \eqref{eq:height} and angle speed \eqref{eq:angle} are given by
\begin{equation}\label{eq:ha}
h(t)=\frac{\Gamma}{8\pi L}(1-2\log c(t))t,
\quad\quad
a(t,\rho)=-\frac{\Gamma_\rho}{2\pi\nu_{\text{tur}}}\log t.
\end{equation}
We remark that the formulas above are first order approximations. In particular, the constant in $h$ could be chosen more accurately by taking into account lower order terms in \eqref{eq:gamma-v}.

In Section \ref{sec:energy} we compute 
the energy of the subsolution. On the one hand, a straightforward computation shows that
$$\frac{1}{2}\int_{\R^3}|\bar{v}(t,x)|^2\dif x=-\frac{L\Gamma^2}{2}\log c(t)+O(1).$$
We remark that the higher order term does not depend on the choice of the vorticity flux. On the other hand, 
$$\frac{3}{2}\int_{\R^3}\lambda_{\max}(\mathring{R})\dif x\lesssim |B_c|||rR||_{L^\infty}=O(1).$$
By combining the contribution of the energy coming from $\bar{v}$ and $R$, we deduce that the energy of the subsolution satisfies \eqref{eq:thm:energy}. The energy of the solutions obtained by convex integration would have slightly more energy, and we can ensure that it is decreasing. 

Once the subsolution is constructed, the true solutions for the Euler equation will be obtained by applying the h-principle  \cite{DeLellisSzekelyhidi10}.
However, since this approach deals with solutions with finite energy, we need to adapt it to the scenario where the energy blows up as time approaches zero. We elaborate on this adaptation in Section \ref{sec:hprinciple}.

\section{Axisymmetric Euler-Reynolds equation}\label{sec:Axisymmetric}

In this section we write the Euler-Reynolds equation \eqref{eq:IER} in the axisymmetric without swirl setting \eqref{ansatz:omegaaxi}-\eqref{ansatz:p}.


\begin{lemma}\label{prop:IER:axi}
Let us suppose that the triple $(\bar{v},\bar{p},R)$ is axisymetric without swirl \eqref{ansatz:omegaaxi}-\eqref{ansatz:p}. Then, the Euler-Reynolds equation \eqref{eq:IER} is rewritten for \eqref{eq:vRaxi} as
\begin{subequations}\label{eq:IERaxi}
	\begin{align}
	\partial_t(r\bar{v}) + \mathrm{div}(r\bar{v}\otimes\bar{v})+r\nabla\bar{p}
	&=-\mathrm{div}(rR),\label{eq:IERaxi:1}\\
	\mathrm{div}(r\bar{v})&=0,\label{eq:IERaxi:2}\\
	\bar{v}|_{t=0}&=v^\circ,
	\end{align}
\end{subequations}
on $[0,T]\times\mathbb{H}$, coupled with the boundary condition
\begin{equation}\label{eq:vr=0}
\bar{v}_r=0\quad\textrm{on}\quad\partial\mathbb{H}.
\end{equation}
\end{lemma}
\begin{proof}
The equivalence between \eqref{eq:IER:2} and \eqref{eq:IERaxi:2} follows from the vector calculus identity
$$
\mathrm{div}_x\bar{v}
=\frac{1}{r}\mathrm{div}_\zeta(r\bar{v}).
$$
We remark that $\bar{v}$ in the left hand side represents the $3$-vector in cartesian coordinates \eqref{ansatz:v}, while $\bar{v}$ in the right hand side represents the $2$-vector in axisymmetric coordinates \eqref{eq:vRaxi}.
The impermeability condition \eqref{eq:vr=0} is equivalent to the continuity of the velocity field in the $x_3$-axis. 
For the equivalence between \eqref{eq:IER:1} and
\eqref{eq:IERaxi:1}, we check that
$$
\mathrm{div}_x R
=\frac{1}{r}\mathrm{div}_\zeta(rR).
$$
We remark that $R$ in the left hand side represents the $3\times 3$ matrix in cartesian coordinates \eqref{ansatz:R}, while $R$ in the right hand side represents the $2\times 2$ matrix in axisymmetric coordinates \eqref{eq:vRaxi}.
Notice that
	\begin{align*}
	e_r\otimes e_r
	&=\frac{1}{r^2}
	\left[\begin{array}{cc|c}
	x_1^2 & x_1x_2 & 0 \\
	x_1x_2 & x_2^2 & 0 \\ \hline
	0 & 0 & 0
	\end{array}\right],\\
	e_r\otimes e_z
	+ e_z\otimes e_r
	&=\frac{1}{r}
	\left[\begin{array}{cc|c}
	0 & 0 & x_1 \\
	0 & 0 & x_2 \\ \hline
	x_1 & x_2 & 0
	\end{array}\right],\\
	e_z\otimes e_z
	&=\left[\begin{array}{cc|c}
	0 & 0 & 0 \\
	0 & 0 & 0 \\ \hline
	0 & 0 & 1
	\end{array}\right].
	\end{align*}
	Hence, by writing $R=(R_1|R_2|R_3)$, we get
	\begin{align*}
	\mathrm{div}_{x}(R_1)
	&=\partial_1\left(R_{rr}\frac{x_1^2}{r^2}\right)
	+\partial_2\left(R_{rr}\frac{x_1x_2}{r^2}\right)
	+\partial_3\left(R_{rz}\frac{x_1}{r}\right)\\
	&=\frac{x_1}{r}\left[\partial_1\left(R_{rr}\frac{x_1}{r}\right)
	+\partial_2\left(R_{rr}\frac{x_2}{r}\right)
	+\partial_3R_{rz}\right]
	+R_{rr}\left[\frac{x_1}{r}\partial_1\left(\frac{x_1}{r}\right)
	+\frac{x_2}{r}\partial_2\left(\frac{x_1}{r}\right)\right]\\
	&=\frac{x_1}{r}\left[\frac{1}{r}\partial_r(rR_{rr})
	+\partial_z R_{rz}\right],\\
	\mathrm{div}_{x}(R_2)
	&=\partial_1\left(R_{rr}\frac{x_1x_2}{r^2}\right)
	+\partial_2\left(R_{rr}\frac{x_2^2}{r^2}\right)
	+\partial_3\left(R_{rz}\frac{x_2}{r}\right)\\
	&=\frac{x_2}{r}\left[\partial_1\left(R_{rr}\frac{x_1}{r}\right)
	+\partial_2\left(R_{rr}\frac{x_2}{r}\right)
	+\partial_3R_{rz}\right]
	+R_{rr}\left[\frac{x_1}{r}\partial_1\left(\frac{x_2}{r}\right)
	+\frac{x_2}{r}\partial_2\left(\frac{x_2}{r}\right)\right]\\
	&=\frac{x_2}{r}\left[\frac{1}{r}\partial_r(rR_{rr})
	+\partial_z R_{rz}\right],
 \end{align*}
 and
 \begin{align*}
	\mathrm{dix}_x(R_3)
	&=\partial_1\left(R_{rz}\frac{x_1}{r}\right)+\partial_2\left(R_{rz}\frac{x_2}{r}\right)+\partial_3 R_{zz}\\
	&=\frac{1}{r}\partial_r(rR_{rz})+\partial_z R_{zz},
	\end{align*}
	where we have applied the vector calculus identity (for $f=R_{rr}$ and $R_{rz}$)
	$$
	\partial_1\left(f\frac{x_1}{r}\right)
	+\partial_2\left(f\frac{x_2}{r}\right)
	=\mathrm{div}_x(fe_r)
	=\frac{1}{r}\partial_r(rf),$$
	and also
	\begin{align*}
	\partial_1\left(\frac{x_1}{r}\right)
	=\frac{x_2^2}{r^3},
	\quad\quad
	\partial_2\left(\frac{x_2}{r}\right)
	=\frac{x_1^2}{r^3}
	\quad\quad
	\partial_1\left(\frac{x_2}{r}\right)
	=\partial_2\left(\frac{x_1}{r}\right)
	=-\frac{x_1 x_2}{r^3}.
	\end{align*}
	Therefore,
	\begin{align*}
	\mathrm{div}_xR
	=\left[\frac{1}{r}\partial_r(rR_{rr})
	+\partial_z R_{rz}\right]e_r
	+\left[\frac{1}{r}\partial_r(rR_{rz})
	+\partial_z R_{zz}\right]e_z.
	\end{align*}
Analogously, the matrix 
$$
\bar{v}\otimes\bar{v}
=\bar{v}_r^2(e_r\otimes e_r)
+\bar{v}_r\bar{v}_z(e_r\otimes e_z+e_z\otimes e_r)
+\bar{v}_z^2(e_z\otimes e_z),
$$ 
satisfies the same formula as $R$.  
\end{proof}

\begin{prop} 
In the setting of Lemma \ref{prop:IER:axi}, the axisymmetric without swirl Euler-Reynolds equation \eqref{eq:IERaxi}\eqref{eq:vr=0} is equivalent to
\begin{subequations}\label{eq:IERaxi:q}
\begin{align}
r(\partial_t \bar{v}-\bar{\omega}_\theta \bar{v}^\perp+\nabla \bar{q})
&=-\mathrm{div}(rR),\label{eq:IERaxi:q:1}\\
\bar{v}|_{t=0}&=v^\circ
\end{align}
\end{subequations}
where
$$\bar{q}=\bar{p}+\frac{1}{2}|\bar{v}|^2$$
is the Bernoulli pressure.
\end{prop}
\begin{proof}
By the Biot-Savart law \eqref{ansatz:omegaaxi}\eqref{ansatz:v}, the conditions \eqref{eq:IERaxi:2} and \eqref{eq:vr=0} are automatically satisfied, while
$$
\bar{\omega}=\mathrm{curl}_x\bar{v}
=(\partial_z\bar{v}_r-\partial_r\bar{v}_z)e_\theta.
$$
Therefore, we have
\begin{equation}\label{eq:divcurlv}
\mathrm{div}_\zeta(r\bar{v})=0,
\quad\quad
\mathrm{curl}_\zeta\bar{v}=-\bar{\omega}_\theta.
\end{equation}
Finally, the equivalence between \eqref{eq:IERaxi:1} and \eqref{eq:IERaxi:q:1} follows from the identity
$$\mathrm{div}_\zeta(r\bar{v}\otimes \bar{v})
=r\bar{v}\cdot\nabla_\zeta \bar{v}
=r\left(\frac{1}{2}\nabla_\zeta|\bar{v}|^2-\bar{\omega}_\theta \bar{v}^\perp\right),$$
where we have applied \eqref{eq:divcurlv}.
\end{proof}




\section{The subsolution}\label{sec:subsolution}

In this section we construct the axisymmetric without swirl subsolution \eqref{ansatz:omegaaxi}-\eqref{ansatz:p}. Firstly, we assume that the vorticity flux \eqref{ansatz:omegaaxi} is given by \eqref{ansatz:gamma}\eqref{ansatz:omegatheta}: $\bar{\omega}_\theta$ is supported in the shifted ball \eqref{eq:shiftedball}, which is parametrized by $\gamma$. The vorticity density must satisfy the compatibility conditions \eqref{compcond1} and \eqref{compcond2}. In the next lemma we show that there exists $\varpi(\rho)$ satisfying these conditions.

\begin{lemma}\label{lem:wprofileintegralconditions}
It is possible to select a vorticity profile $\varpi\in C(\R^+)$, supported on $(0,1)$, and satisfying \eqref{compcond1} and \eqref{compcond2}.
\end{lemma}
\begin{proof}
For instance, we can make the ansatz for $0\leq\rho\leq 1$
$$\varpi(\rho)=\Gamma c_1\rho(1-\rho)\big(1-c_2\log(e\rho)\big),$$
and adjust the parameters $c_1,c_2\in \R$. Then, we extend $\varpi(\rho)=0$ outside the interval $[0,1]$. 
\end{proof}

\begin{Rem}
The particular choice of $\varpi(\rho)$ in the previous lemma is not important in this work. In fact, we will show in Section \ref{sec:energy} that the leading term in the energy does not depend on $\varpi(\rho)$. We will make all computations for any $\varpi(\rho)$ that satisfies Lemma \ref{lem:wprofileintegralconditions}.
\end{Rem}

\subsection{The velocity}\label{sec:velocity}

In this section we give the velocity by means of the axisymmetric without swirl Biot-Savart law.

\begin{lemma}[Axisymmetric without swirl Biot-Savart law]\label{lemma:BSlaw} It holds
$$\bar{v}(t,\zeta)
=\int_{\mathbb{H}}
K^{\mathrm{ax}}(\zeta,\zeta')
\bar{\omega}_\theta(t,\zeta')\dif\zeta',$$
for $\zeta=(r,z)\in\mathbb{H}$, where
$$K^{\text{ax}}(\zeta,\zeta')
=\frac{i}{2\pi r}
\left(\sqrt{\frac{r'}{r}}H(s)-\bar{\zeta} G(s)\right),$$
and
\begin{align*}
H(s)=\int_0^\pi
\frac{1-\cos\varphi}{(2(1-\cos\varphi)+s)^{\nicefrac{3}{2}}}\dif\varphi,\\
G(s)=\int_0^\pi
\frac{\cos\varphi}{(2(1-\cos\varphi)+s)^{\nicefrac{3}{2}}}\dif\varphi,
\end{align*}
for the variables
$$\bar{\zeta}=\frac{\zeta-\zeta'}{\sqrt{rr'}},
\quad\quad
s=|\bar{\zeta}|^2.$$
\end{lemma}
\begin{proof} See Section \ref{sec:ABS}.
\end{proof}

Next, we derive asymptotic formulas for the Biot-Savart operator evaluated on $\gamma$
$$\bar{v}(t,\gamma)
=\int_0^{2\pi}\int_0^1K^{\text{ax}}(\gamma,\gamma')\varpi(\rho')\rho'\dif\rho'\dif\alpha',$$
where we have abbreviated $\gamma=\gamma(t,\rho,\alpha)$ and $\gamma'=\gamma(t,\rho',\alpha')$.

To obtain the leading order terms, we will assume that $c\ll 1$. Our $c$ will be continuous and null at $t=0$, so this condition is satisfied for short times.

\begin{prop}\label{prop:Birkhoff-Rott}
It holds
$$\bar{v}(t,\gamma)=
-\frac{i\rho\Gamma_\rho}{2\pi c}e^{i(\alpha+a)}
-\frac{i\Gamma}{4\pi L}\log c
+O(1),$$
where $\Gamma_\rho$ is defined in equation \eqref{eq:Gamma_rhodefinition}.
\end{prop}
\begin{proof}
We follow the notation of Lemma \ref{lemma:BSlaw}. We decompose
$$\bar{v}
=\bar{v}_{\circlearrowright}
+\bar{v}_{\uparrow},$$
where
\begin{align*}
\bar{v}_{\circlearrowright}
&=-\frac{i}{2\pi r}
\int_{\mathbb{H}}
\bar{\zeta} G(s)\bar{\omega}_\theta\dif\zeta',\\
\bar{v}_{\uparrow}
&=\frac{i}{2\pi r}\int_{\mathbb{H}}
\sqrt{\frac{r'}{r}}
H\bar{\omega}_\theta\dif\zeta'.
\end{align*} 
Firstly, we compute $\bar{v}_{\circlearrowright}$. Notice that we can write
$$\bar{\zeta}=c\frac{\rho-\rho' e^{i\bar{\alpha}}}{\sqrt{rr'}}e^{i(\alpha+a)},
\quad\quad
s=c^2
\frac{|\rho-\rho' e^{i\bar{\alpha}}|^2}{rr'},$$
where $\bar{\alpha}(t,\alpha,\alpha',\rho,\rho')=\alpha'+a(t,\rho')-\alpha-a(t,\rho)$.
From Lemma \ref{lemma:Gestimates}, we know that
$$G(s)
=\frac{1}{s}
+O(\log s),$$
as $s\to 0$.
Hence,
\begin{align*}
\bar{v}_{\circlearrowright}
&=-\frac{i}{2\pi r}\int_0^1\int_0^{2\pi}
\bar{\zeta} G(s)\dif\alpha'\varpi(\rho')\rho'\dif\rho'\\
&=-\frac{i}{2\pi r}
\int_0^1\int_0^{2\pi}
\bar{\zeta}\left(\frac{1}{s}+O(\log s)\right)\dif\alpha'\varpi(\rho')\rho'\dif\rho'\\
&=-\frac{i}{2\pi r}
\int_0^1\int_0^{2\pi}
\frac{1}{\bar{\zeta}^*}\dif\alpha'\varpi(\rho')\rho'\dif\rho'
+\frac{i}{2\pi r}
\int_0^1\int_0^{2\pi}
\bar{\zeta} O(\log s)\dif\alpha'\varpi(\rho')\rho'\dif\rho'.
\end{align*}
Since $|\bar{\zeta}|=\sqrt{s}$, the integrand in the second term is bounded, so the result is $O(c\log c)=O(1)$. The first term can be computed ($\rho'=\lambda\rho$)
\begin{align*}
-\frac{i}{2\pi r}
\int_0^1\int_0^{2\pi}
\frac{\dif\alpha'}{\bar{\zeta}^*}\varpi(\rho')\rho'\dif\rho'
=&-\frac{i}{2\pi c}e^{i(\alpha+a)}
\int_0^1\int_0^{2\pi}
\frac{\dif\bar{\alpha}}{\rho-\rho'e^{-i\bar{\alpha}}}\varpi(\rho')\rho'\dif\rho'\\
&-\frac{i}{2\pi rc}e^{i(\alpha+a)}
\int_0^1\int_0^{2\pi}
\frac{(\sqrt{rr'}-r)\dif\bar{\alpha}}{\rho-\rho'e^{-i\bar{\alpha}}}\varpi(\rho')\rho'\dif\rho'
&&=:I\\
=&-\frac{i}{c}e^{i(\alpha+a)}
\int_0^{\rho}\varpi(\rho')
\frac{\rho'}{\rho}\dif\rho'
+O(1)\\
=&-\frac{i}{c}e^{i(\alpha+a)}
\int_0^{1}\varpi(\rho\lambda)
\rho\lambda\dif\lambda
+O(1),
\end{align*}
where we have applied
$$\frac{1}{2\pi}\int_0^{2\pi}
\frac{\dif\bar{\alpha}}{\rho-\rho'e^{-i\bar{\alpha}}}
=\frac{1}{2\pi}\int_0^{2\pi}
\frac{e^{i\bar{\alpha}}\dif\bar{\alpha}}{\rho e^{i\bar{\alpha}}-\rho'}
=\frac{1}{2\pi\rho i}
\int_{\partial B_\rho}
\frac{\dif z}{z-\rho'}
=\frac{1}{\rho}\mathbf{1}_{\rho>\rho'},$$
and
$$
\begin{aligned}
|I|
&\lesssim 
\left|\frac{1}{\sqrt{r}}
\int_0^1\int_0^{2\pi}
\frac{\rho \cos(a+\alpha)-\rho'\cos(a+\alpha')}{(\rho-\rho'e^{-i\bar{\alpha}})(\sqrt{r'}+\sqrt{r})}\varpi(\rho')\rho'\dif\bar{\alpha}\dif\rho'\right|\\
&\lesssim\left|
\int_0^1\int_0^{2\pi}
O(1)\varpi(\rho')\rho'\dif\bar{\alpha}\dif\rho'\right|= O(1).
\end{aligned}
$$
Therefore,
$$\bar{v}_{\circlearrowright}
=-\frac{i\rho\Gamma_\rho}{2\pi c}e^{i(\alpha+a)}
+O(1).$$
Secondly, we compute $\bar{v}_{\uparrow}$. From Lemma \ref{lemma:Hestimates}, we know that
$$
H(s)=-\frac{1}{4}\log s+O(1),
$$
as $s\to 0$. Hence,
\begin{align*}
\bar{v}_{\uparrow}
&=\frac{i}{2\pi r}\int_0^1\int_0^{2\pi}
\sqrt{\frac{r'}{r}}
H\dif\alpha'\varpi(\rho')\rho'\dif\rho'\\
&=\frac{i}{2\pi r}\int_0^1\int_0^{2\pi}
\sqrt{\frac{r'}{r}}
\left(-\frac{1}{4}\log s+O(1)\right)\dif\alpha'\varpi(\rho')\rho'\dif\rho'\\
&=-\frac{i}{8\pi r}\int_0^1\int_0^{2\pi}
\sqrt{\frac{r'}{r}}
\log s\dif\alpha'\varpi(\rho')\rho'\dif\rho'+O(1).
\end{align*}
We further compute
\begin{align*}
\bar{v}_{\uparrow}
=&-\frac{i}{8\pi r}\int_0^1\int_0^{2\pi}
\sqrt{\frac{r'}{r}} \left(2\log c+\log\left(\frac{|\rho-\rho'e^{i\bar{\alpha}}|^2}{rr'}\right)\right)\dif\alpha'\varpi(\rho')\rho'\dif\rho'+O(1)\\
=&-\frac{i}{4\pi r}\int_0^1\int_0^{2\pi}
\sqrt{\frac{r'}{r}}\log c\dif\alpha'\varpi(\rho')\rho'\dif\rho'+O(1)\\
&+\frac{i}{2\pi r}\int_0^1\int_0^{2\pi}\sqrt{\frac{r'}{r}}\log |\rho-\rho'e^{i\bar{\alpha}}|\dif\alpha'\varpi(\rho')\rho'\dif\rho'+O(1)
\\=&-\frac{i\Gamma}{4\pi L}\log c+O(1)\int_{B(0,2)}|\log{|x|}|dx+O(1)=-\frac{i\Gamma}{4\pi L}\log c+O(1),
\end{align*}
where we have used that $r=L+O(c)$.
\end{proof}


\subsection{The pressure}\label{sec:pressure}

In this section we give the pressure by means of the Bernoulli law.

\begin{lemma}[The Bernoulli law]
There exists $\bar{q}_0$ satisfying
$$\partial_t \bar{v}+\nabla \bar{q}_0
=\bar{\omega}_\theta\partial_t\gamma^\perp.$$
\end{lemma}
\begin{proof}
Let us denote $\bar{v}^{\text{ax}}=\bar{v}$, and define
$$
\bar{v}^{2d}(t,\zeta)
=\int_{\R^2}K^{2d}(\zeta-\zeta')\bar{\omega}_\theta(t,\zeta')\dif\zeta',
$$
where
$$
K^{2d}(\zeta)=\frac{i}{2\pi\zeta^*}
$$
is the $2d$ Biot-Savart kernel.
Since 
$$
\mathrm{curl_\zeta}
(\bar{v}^{2d})
=\bar{\omega}_\theta
=-\mathrm{curl_\zeta}
(\bar{v}^{\text{ax}})$$ 
in the simply-connected domain $\mathbb{H}$, there exists $\tilde{q}$ satisfying
$$
\partial_t(\bar{v}^{\text{ax}}+\bar{v}^{2d})
=\nabla\tilde{q}.$$
We claim that 
$$\bar{q}_{0,0}(t,\zeta)
=-\int_0^{2\pi}\int_0^1
K^{2d}(\zeta-\gamma)\cdot \partial_t\gamma\varpi\rho\dif\rho\dif\alpha,$$
where $\gamma=\gamma(t,\rho,\alpha)$
satisfies
\begin{equation}\label{claim:pressure}
\partial_t\bar{v}^{2d}
+\bar{\omega}_\theta\partial_t\gamma^\perp
=\nabla\bar{q}_{0,0}.
\end{equation}
Then, the statement holds for 
$$
\bar{q}_0=\bar{q}_{0,0}-\tilde{q}.
$$
Let $\phi\in C_c^\infty((0,T)\times\mathbb{H})$ be a test function. On the one hand,
\begin{align*}
\langle\partial_t\bar{v}^{2d},\phi\rangle
&=-\langle\bar{v}^{2d},\partial_t\phi\rangle
=-\int_0^T\int_{\mathbb{H}}\bar{v}^{2d}\partial_t\phi\dif\zeta\dif t\\
&=-\int_0^T\int_{\mathbb{H}}\int_0^{2\pi}\int_0^1K^{2d}(\zeta-\gamma)\varpi\rho\dif\rho\dif\alpha
\partial_t\phi\dif\zeta\dif t\\
&=-\int_0^{2\pi}\int_0^1 \lim_{\varepsilon\to 0}\left(\int_0^T\int_{\mathbb{H}\setminus B_\varepsilon(\gamma)}K^{2d}(\zeta-\gamma)\partial_t\phi\dif\zeta\dif t\right)\varpi\rho\dif\rho\dif\alpha.
\end{align*}
We parameterize the (space-time) cylinder $\partial B_\varepsilon(\gamma)$ by
$$
X(t,\beta)
=(t,\gamma(t)+\varepsilon e^{i\beta}).
$$
The normal vector to the boundary pointing into the interior of the cylinder is given by
$$
n=\frac{\partial_t X\times\partial_\beta X}{|\partial_t X\times\partial_\beta X|},
\quad\quad
\partial_t X\times\partial_\beta X
=\varepsilon(\partial_t\gamma\cdot(e^{i\beta}),-e^{i\beta}).
$$
Therefore, an integration by parts yields
\begin{align}
\int_0^T\int_{\mathbb{H}\setminus B_\varepsilon(\gamma)}K^{2d}(\zeta-\gamma)\partial_t\phi\dif\zeta\dif t
=&\int_0^T\int_{\partial B_\varepsilon(\gamma)}K^{2d}(\zeta-\gamma)\phi n_t\dif\zeta\dif t\label{eq:vpartscyl1}\\
&-\int_0^T\int_{\mathbb{H}\setminus B_\varepsilon(\gamma)}\partial_t(K^{2d}(\zeta-\gamma))\phi\dif\beta\dif t.\label{eq:vpartscyl2}
\end{align}
The first term equals
\begin{align*}
\eqref{eq:vpartscyl1}
&=\int_0^T\int_0^{2\pi}\varepsilon K^{2d}(\varepsilon e^{i\beta})\phi(t,\gamma+\varepsilon e^{i\beta}) \partial_t\gamma\cdot(e^{i\beta})\dif\zeta\dif t.
\end{align*}
Since
$$
\varepsilon K^{2d}(\varepsilon e^{i\beta})
=\frac{ie^{i\beta}}{2\pi},
$$
we get
$$
\lim_{\varepsilon\to 0}
\eqref{eq:vpartscyl1}
=\frac{1}{2\pi}\int_0^T\phi(t,\gamma)\int_0^{2\pi}ie^{i\beta}\partial_t\gamma\cdot(e^{i\beta})\dif\beta\dif t
=\frac{1}{2}\int_0^T\phi(t,\gamma)\partial_t\gamma^\perp\dif t.
$$
For the second term \eqref{eq:vpartscyl2} we have
$$
\partial_t(K^{2d}(\zeta-\gamma))
=-D_{\zeta}K^{2d}(\zeta-\gamma)\partial_t\gamma.
$$
By plugging everything together we obtain
\begin{align*}
\langle\partial_t\bar{v}^{\text{2d}},\phi\rangle
=&-\int_0^{2\pi}\int_0^1 \lim_{\varepsilon\to 0}\left(\int_0^T\int_{\mathbb{H}\setminus B_\varepsilon(\gamma)}D_{\zeta}K^{\text{2d}}(\zeta-\gamma)\partial_t\gamma\phi\dif\zeta\dif t\right)\varpi\rho\dif\rho\dif\alpha\\
&-\frac{1}{2}\int_0^{2\pi}\int_0^1\left(\int_0^T\phi(t,\gamma)\partial_t\gamma^\perp\dif t\right)\varpi\rho\dif\rho\dif\alpha.
\end{align*}
On the other hand,
\begin{align*}
\langle \nabla \bar{q}_{0,0},\phi\rangle
&=-\langle \bar{q}_{0,0},\nabla\phi\rangle
=-\int_0^T\int_{\mathbb{H}} \bar{q}_{0,0}\nabla\phi\dif\zeta\dif t\\
&=\int_0^T\int_{\mathbb{H}}\int_0^{2\pi}\int_0^1 K^{2d}(\zeta-\gamma)\cdot \partial_t\gamma\varpi \rho\dif\rho\dif\alpha\nabla\phi\dif\zeta\dif t\\
&=\int_0^{2\pi}\int_0^1\lim_{\varepsilon\to 0}\left(\int_0^T\int_{\mathbb{H}\setminus B_\varepsilon(\gamma)} K^{2d}(\zeta-\gamma)\cdot\partial_t\gamma\nabla\phi\dif\zeta\dif t\right)\varpi\rho\dif\rho\dif\alpha.
\end{align*}
An integration by parts yields
\begin{align}
\int_0^T\int_{\mathbb{H}\setminus B_\varepsilon(\gamma)} K^{2d}(\zeta-\gamma)\cdot\partial_t\gamma\nabla\phi\dif\zeta\dif t
=&\int_0^T\int_{\partial B_\varepsilon(\gamma)} K^{2d}(\zeta-\gamma)\cdot\partial_t\gamma\phi n_\zeta\dif\zeta\dif t\label{eq:qpartscyl1}\\
&-\int_0^T\int_{\mathbb{H}\setminus B_\varepsilon(\gamma)} \nabla(K^{2d}(\zeta-\gamma)\cdot\partial_t\gamma)\phi\dif\zeta\dif t.\label{eq:qpartscyl2}
\end{align}
Since the first term equals
$$
\eqref{eq:qpartscyl1}
=\int_0^T\int_0^{2\pi} \varepsilon K^{2d}(\varepsilon e^{i\beta})\cdot\partial_t\gamma\phi(t,\gamma+\varepsilon e^{i\beta}) (-e^{i\beta})\dif\beta\dif t,
$$
we get
$$
\lim_{\varepsilon\to 0}\eqref{eq:qpartscyl1}
=-\frac{1}{2\pi}\int_0^T\phi(t,\gamma)\int_0^{2\pi}(ie^{i\beta})\cdot\partial_t\gamma e^{i\beta}\dif\beta\dif t
=\frac{1}{2}\int_0^T\phi(t,\gamma) \partial_t\gamma^\perp\dif t.
$$
For the second term \eqref{eq:qpartscyl2}, we have
$$
\nabla(K^{2d}\cdot\partial_t\gamma)
=(D_{\zeta}K^{2d})^\dagger\partial_t\gamma
=D_{\zeta}K^{2d}\partial_t\gamma
+(\mathrm{curl}_\zeta K^{2d})\partial_t\gamma^\perp.
$$
Since $\mathrm{curl}_\zeta K^{2d}=\delta_{\zeta'}$, we indeed have
$$
\nabla(K^{2d}\cdot\partial_t\gamma)
=D_{\zeta}K^{2d}\partial_t\gamma,
$$
for $\zeta\in\mathbb{H}\setminus B_\varepsilon(\gamma)$. 
By plugging everything together, we obtain
\begin{align*}
\langle \nabla \bar{q}_{0,0},\phi\rangle
=&-\int_0^{2\pi}\int_0^1
\lim_{\varepsilon\to 0}
\left(\int_0^T\int_{\mathbb{H}\setminus B_\varepsilon(\gamma)}D_{\zeta}K^{\text{2d}}(\zeta-\gamma)\partial_t\gamma\phi\dif\zeta\dif t\right)\varpi\rho\dif\rho\dif\alpha\\
&+\frac{1}{2}\int_0^{2\pi}\int_0^1\left(\int_0^T\phi(t,\gamma)\partial_t\gamma^\perp\right)\varpi\rho\dif\rho\dif\alpha.
\end{align*}
Finally, we get
$$
\langle \partial_t\bar{v}-\nabla\bar{q}_{0,0},\phi\rangle
=-\int_0^T\int_0^{2\pi}\int_0^1\phi(t,\gamma)\partial_t\gamma^\perp \varpi\rho\dif\rho\dif\alpha\dif t,
$$
that is, our claim \eqref{claim:pressure}.
\end{proof}

Finally, we take the Bernoulli pressure as
$$\bar{q}
=\bar{q}_0+\bar{q}_1,$$
where $\bar{q}_1$ is a corrector introduced to guarantee several compatibility conditions that appear in the next section. We take $\bar{q}_1$ smooth and with the same support as $\bar{\omega}$.
Therefore,
\begin{equation}\label{eq:Bernoulliq1}
\partial_t \bar{v}+\nabla \bar{q}
=\bar{\omega}_\theta\partial_t\gamma^\perp 
+\nabla \bar{q}_1.
\end{equation}

\subsection{The Reynolds stress}\label{sec:Reynolds}

In this section we give the Reynolds stress by means of the Isett-Oh antidivergence operator \cite{IsettOh16}. In view of \eqref{eq:Bernoulliq1}, 
the Euler-Reynolds equation \eqref{eq:IERaxi:q} reads as
\begin{equation}\label{eq:Reynolds}
-\mathrm{div}(rR)
=F,
\end{equation}
where we have abbreviated
\begin{equation}\label{eq:F}
\begin{split}
F
&=\partial_t(r\bar{v})
+\mathrm{div}(r\bar{v}\otimes\bar{v})
+r\nabla\bar{p}\\
&=r(\bar{\omega}_\theta(\partial_t\gamma-\bar{v})^\perp
+\nabla \bar{q}_1)\\
&=r\left(\frac{\varpi}{c^2}\left(\dot{c}\rho e^{i(\alpha+a)}+O(1)\right)+\nabla\bar{q}_1\right).
\end{split}
\end{equation}
As we anticipated in Sections \ref{sec:Heuristics} and \ref{sec:sketch}, our goal is to confine $R$ inside the support of $\bar{\omega}$. To this end, we use the antidivergence operator in \cite{IsettOh16}, whose 
construction is recalled in Appendix \ref{section:antidivergence} for the convenience of the reader. This operator requires the compatibility conditions
\begin{equation}\label{eq:compatibilityconditionsantidivergence}
\int_{\mathbb{H}} F\dif\zeta=0,
\quad\quad
\int_{\mathbb{H}}F\cdot(\zeta-\zeta_0)^\perp\dif\zeta=0, 
\end{equation} 
where $\zeta_0=L+ih$.
The first condition arises from integrating \eqref{eq:Reynolds}. The second condition stems from imposing that $R$ is symmetric.
In the next lemma, we show that we can choose the corrected pressure in such a way that these conditions are satisfies. During the proof, we will apply Lemma \ref{lem:wprofileintegralconditions} to eliminate the higher order terms.

\begin{lemma} We can select $\bar{q}_1$ smooth, with the same support as $\bar{\omega}_\theta$, satisfying the compatibility conditions \eqref{eq:compatibilityconditionsantidivergence} and
\begin{equation}\label{eq:qrelatedtoF}
\nabla \bar{q}_1= O(c^{-3}),
\end{equation}
as $t\to 0$.
\end{lemma}
\begin{proof}
We fix two functions $\phi_1$ and $\phi_2$ supported on $\gamma$ such that 
$$
\phi_j(\gamma(t,\rho,\alpha))=\frac{1}{c(t)^2}\bar{\phi}_j(\rho)
\quad\text{with}\quad
2\pi\int_0^1\bar{\phi}_j(\rho)\rho^j \dif\rho=1.$$ 
We split $\bar{q}_1=\bar{q}_{1,1}+\bar{q}_{1,2}$ into
$$
\bar{q}_{1,1}
=c_1\phi_1,
\quad\quad
\bar{q}_{1,2}
=
c_2\phi_2 \sin (\alpha+a),
$$
for some time dependent constants $c_1$ and $c_2$, to be determined.
On the one hand, an integration by parts yields
$$\int_{\mathbb{H}}
r\nabla \bar{q}_1
=-e_r\int_{\mathbb{H}} \bar{q}_1
=-e_rc^2\int_0^{2\pi}\int_0^1\bar{q}_{1}\rho\dif\rho\dif\alpha
=-c_1e_r.$$
It is easy to check that the integral of the other terms in \eqref{eq:F} are also proportional to $e_r$. In fact, an integration by parts yields
$$\int_{\mathbb{H}}\partial_t(r\bar{v})
=\partial_t\int_{\mathbb{H}}\nabla^\perp\psi
=-e_z\partial_t\int_{\partial\mathbb{H}}\psi
=0,$$
where $\psi$ is the stream function, 
and also
$$\int_{\mathbb{H}}\mathrm{div}(r\bar{v}\otimes \bar{v})=0,$$
so only the pressure term contributes to the mean.
Hence, it is possible to select $c_1$ canceling the mean. 
We can compute
$$
\begin{aligned}
A_1(t)=\int_{\mathbb{H}}r\bar{\omega}_\theta(\partial_t\gamma-\bar{v})^\perp\dif\zeta&=i\dot{c}\int_0^1\int_0^{2\pi}r\varpi\rho^2 e^{i(\alpha+a)}\dif \alpha \dif \rho+\int_{\mathbb{H}}r\bar{\omega}_\theta O(1)\dif\zeta\\
&=ic\dot{c}\int_0^1\int_0^{2\pi}\varpi\rho^3\cos(\alpha+a)e^{i(\alpha+a)}\dif\alpha\dif\rho+O(1)\\
&=i\pi c\dot{c}\int_0^1\varpi \rho^3 \dif\rho+O(1)=
O(1),
\end{aligned}
$$
where the last integral vanishes by Lemma \ref{lem:wprofileintegralconditions}.
Hence, the first condition in \eqref{eq:compatibilityconditionsantidivergence} is satisfied by taking $c_1=A_1\cdot e_r=O(1)$. Notice that $|\nabla\bar{q}_{1,1}|\lesssim c^{-3}.$ On the other hand, an integration by parts yields
$$\int_{\mathbb{H}}r\nabla^\perp \bar{q}_{1}\cdot(\zeta-\zeta_0)\dif\zeta
=-\int_{\mathbb{H}}(z-h)\bar{q}_{1}
=-c^3\int_0^{2\pi}\int_0^1\bar{q}_{1}\sin(\alpha+a)\rho^2\dif\rho\dif\alpha
=-c\frac{c_2}{2}.$$
We have that
\begin{align*}
A_2(t)&=\int_{\mathbb{H}}r\bar{\omega}_\theta(\partial_t\gamma-v)\cdot(\zeta-\zeta_0)\dif\zeta\\
&=c\dot{c}\int_0^1\int_0^{2\pi}r\varpi(e^{i(\alpha+a)})\cdot(e^{i(\alpha+a)})\rho^3\dif\alpha\dif\rho
+\int_{\mathbb{H}}r\bar{\omega}_\theta O(1)\cdot(\zeta-\zeta_0)\dif\zeta\\
&=c\dot{c}\int_0^1\int_0^{2\pi}(L+c\rho\cos(\alpha+a))\varpi\rho^3\dif\alpha\dif\rho+O(c)\\
&=2\pi Lc\dot{c}\int_0^1\varpi\rho^3\dif\rho+O(c)
=O(c),
\end{align*}
where the last integral vanishes by Lemma \ref{lem:wprofileintegralconditions}.
Hence, the second condition in \eqref{eq:compatibilityconditionsantidivergence} is satisfied by taking $c_2=-\frac{2A_2}{c}=O(1)$. Notice that $|\nabla\bar{q}_{1,2}|\lesssim c^{-3}.$
\end{proof}

Once $\bar{q}_1$ is determined, we can use the antidivergence operator from Appendix \ref{section:antidivergence} to define
\begin{equation}\label{eq:R}
R=-\frac{1}{r}\mathcal{R}(F).
\end{equation}
By Lemma \ref{lem:antidivergence}, it follows that
$$
\|R\|_{L^\infty}
\lesssim c
\|F\|_{L^\infty}
\lesssim\frac{c\dot{c}+1}{c^2},
$$
for short times.

\section{The energy}\label{sec:energy}
In this section, we compute the energy of the subsolution. We keep our choices of $c$, $h$ and $a$ from \eqref{eq:thickness} and \eqref{eq:ha}, respectively. The main goal is to conclude that the total energy is given by the formula
\begin{equation}\label{eq:Esub}
E_{\text{sub}}(t)
=\int_{\R^3}e_{\text{sub}}(\bar{v},\mathring{R})\dif x
=-\frac{L\Gamma^2}{2}\log c(t) +O(1).
\end{equation}
We split the total energy into 
$$E_{\text{sub}}=E_{\bar{v}}+E_R,$$
where
\begin{align*}
E_{\bar{v}}(t)
&=\frac{1}{2}\int_{\R^3}|\bar{v}(t,x)|^2\dif x,\\
E_R(t)
&=\frac{3}{2}\int_{\R^3}\lambda_{\max}(\mathring{R}(t,x))\dif x.
\end{align*}

\begin{lemma}
It holds
$$
E_{\bar{v}}(t)
=-\frac{L\Gamma^2}{2}\log c(t)+O(1),
$$
as $t\to 0$.
\end{lemma}
\begin{proof}
Firstly, we recall that, since the vorticity $\bar{\omega}$ is compactly supported and has zero mean, the velocity $\bar{v}$ decays like $|x|^{-3}$ as $|x|\to\infty$. For positive times, $\bar{v}$ is also locally bounded. Hence, we have $\bar{v}(t)\in L^2(\R^3)$ for $t>0$.

Next, we proceed to compute the energy for positive times. 
Instead of integrating directly in space, by a direct computation we obtain
$$\partial_t E_{\bar{v}}=\int_{\R^3} \bar{v}\cdot\partial_t\bar{v}\dif x=-\int_{\R^3}\bar{v}\cdot \mathrm{div}_x R \dif x.
$$
By changing variables and using \eqref{eq:Reynolds} and $\mathrm{div}(r\bar{v})=0$, we obtain 
\begin{align*}
\partial_t E_{\bar{v}}
&=-\int_{\R^3}\bar{v}\cdot \frac{1}{r}\mathrm{div}_{\zeta} (rR) \dif x= \frac{2\pi}{c^2}\int_{\mathbb{H}}r\varpi\bar{v}\cdot(\partial_t\gamma-\bar{v})^\perp \dif\zeta\\
&=2\pi\int_0^{2\pi}\int_0^1r\varpi(\rho)\bar{v}\cdot(\partial_t\gamma-\bar{v})^\perp \rho \dif \rho \dif\alpha.
\end{align*}
From equations \eqref{eq:vupcircle} and \eqref{eq:gamma-v:1}, we see that the higher order term in $c$ comes from 
\begin{equation}\label{eq:energycomputation}
\begin{aligned}
2\pi\int_0^{2\pi}\int_0^1\bar{v}_{\circlearrowright}\cdot \frac{\dot{c}\rho}{2} e^{i(\alpha+a)} r\varpi(\rho)\rho \dif\rho \dif\alpha
&=-\frac{2\pi \dot{c}}{c}\int_0^{2\pi}\int_0^1r\varpi(\rho)\rho^3\Gamma_\rho \dif\rho \dif\alpha\\
&=-4\pi^2L\frac{\dot{c}}{c}\int_0^1\varpi(\rho)\rho^3\Gamma_\rho \dif\rho.
\end{aligned}
\end{equation}
We rewrite \eqref{eq:Gamma_rhodefinition} as
$$\Gamma_\rho=\frac{1}{\rho^2}\int_0^\rho \varpi(\rho')\rho'd\rho'=\frac{1}{\rho^2}f(\rho),$$
obtaining 
$$\int_0^1 \varpi(\rho)\rho^3\Gamma_\rho \dif\rho=\int_0^1 f'(\rho)f(\rho)d\rho=\frac{1}{2}(f^2(1)-f^2(0))=\frac{\Gamma^2}{2(2\pi)^2},$$
independently of $\varpi$.
Thus, 
$$
\eqref{eq:energycomputation}
=-\frac{L\Gamma^2}{2}\frac{\dot{c}}{c}.
$$
We can do similar computations for the remaining terms to obtain
$$\partial_t E_{\bar{v}}=-\frac{L\Gamma^2}{2}\partial_t(\log c)+O(\dot{c}\log c).$$
Finally, notice that the last term is integrable in time.
\end{proof}

\begin{lemma}
It holds,
$$
E_R(t)=O(1),
$$
as $t\to 0$.
\end{lemma}
\begin{proof}
Notice that 
$$
\lambda_{\max}(\mathring{R})
\lesssim
\|\mathring{R}\|_{L^\infty}
\lesssim\|R\|_{L^\infty}
\lesssim\|rR\|_{L^\infty}
\lesssim c\|F\|_{L^\infty},
$$
where we have applied
Lemma \ref{lem:antidivergence} in the last inequality. 
From equations \eqref{eq:gamma-v:1} and \eqref{eq:qrelatedtoF}, we deduce that 
$$||F||_{L^\infty}\lesssim \frac{c\dot{c}+1}{c^3},$$
concluding that
$$E_R=\int_{\R^3}\frac{3}{2}\lambda_{\max}(\mathring{R})\lesssim |B_c|\, ||rR||_{L^\infty}=O(c\dot{c}+1),$$
which is $O(1)$ by our choice of $c$.
\end{proof}

\section{Time-weighted h-principle}\label{sec:hprinciple}

In this section, we show how the De Lellis-Sz\'ekelyhidi h-principle for the Euler equation \cite{DeLellisSzekelyhidi10} can be adapted for constructing velocities in the space $C_{\log t}L^2$. We define this space, for any $1\leq p<\infty$, by the
time-weighted norm
$$
\|f\|_{C_{\log t}L^p}
:=\sup_{t\in (0,T]}\frac{\|f(t)\|_{L^p}}{\max\{1,|\log t|\}^{\nicefrac{1}{p}}}<\infty.
$$
Firstly, it is convenient to rewrite the Euler-Reynolds equation \eqref{eq:IER} in linear form 
\begin{subequations}\label{eq:IER:S}
	\begin{align}
	\partial_t\bar{v} + \mathrm{div}\bar{S}+\nabla\bar{\pi}
	&=0,\label{eq:IERS:1}\\	\mathrm{div}\bar{v}&=0,\label{eq:IERS:2}\\
	\bar{v}|_{t=0}&=v^\circ,
	\end{align}
\end{subequations}
where we have made the change of variables
$$
\bar{\pi}=\bar{p}+\frac{1}{3}(|\bar{v}|^2+\mathrm{tr}R),
\quad\quad
\bar{S}
=\bar{v}\ocircle\bar{v}+\mathring{R},
$$
with
$$
\bar{v}\ocircle\bar{v}
=\bar{v}\otimes\bar{v}
-\frac{1}{3}|\bar{v}|^2I_3,
\quad\quad
\mathring{R}
=R-\frac{1}{3}(\mathrm{tr}R)I_3.
$$
Thus, a triple
\begin{equation}\label{eq:subsolution}
(\bar{v},\bar{\pi},\bar{S})\in C_{\log t}(L^2\times L^1_{loc}\times L^1)
\end{equation}
is a weak solution to the Euler-Reynolds equation \eqref{eq:IER:S} if $\bar{v}$ is a weakly divergence-free vector field, $\bar{S}$ is a traceless symmetric tensor, and $\bar{\pi}$ is a scalar pressure, such that
\begin{equation}\label{eq:weakmomentum}
\int_0^T\int_{\R^2}
(\bar{v}\cdot\partial_t\varphi
+\bar{S}:\nabla\varphi+\bar{\pi}\mathrm{div}\varphi)(t,x)\dif x\dif t
+\int_{\R^2}v^\circ(x)\varphi(0,x)\dif x=0,
\end{equation}
for all test function $\varphi\in C_c^1([0,T)\times\R^2)$.
Notice that the integrability condition \eqref{eq:subsolution} is enough to make the expression \eqref{eq:weakmomentum} meaningful.

Let $(\bar{v},\bar{\pi},\bar{S})$
be the subsolution we constructed in the previous sections.
The associated energy density is given by the convex function 
(see \cite[Lemma 3]{DeLellisSzekelyhidi10})
$$
e_{\text{sub}}(\bar{v},\bar{S})
=\frac{3}{2}\lambda_{\max}(\bar{v}\otimes\bar{v}-\bar{S})
=\frac{1}{2}|\bar{v}|^2
+\frac{3}{2}\lambda_{\max}(\mathring{R}).
$$
The total energy $E_{\text{sub}}$ satisfies \eqref{eq:Esub}.
We define the energy density of the (true) solutions by
$$
e=e_{\text{sub}}(\bar{v},\bar{S})+\tilde{e},
$$
where $\tilde{e}$ is a smooth energy corrector. We take $\tilde{e}$ vanishing outside $\Omega_{\text{tur}}$, strictly positive inside, and in such a way that
the corresponding energy
\begin{equation}\label{eq:trueenergy}
E(t)=\int_{\R^3} e(t,x)\dif x=E_{\text{sub}}(t) +O(1),
\end{equation}
is strictly decreasing.

We define the subspace $X_0$ of velocities $v\in C_{\log t}L^2$ satisfying:
\begin{enumerate}[(i)]
    \item\label{sub:i} There exists $S$ such that 
    $(v,\bar{\pi},S)$
    is a weak solution to the Euler-Reynolds equation.
    \item\label{sub:ii} $(v,S)=(\bar{v},\bar{v}\ocircle\bar{v})$ outside $\Omega_{\text{tur}}$.
    \item\label{sub:iii} $(v,S)$ is continuous and satisfy $e_{\text{sub}}(v,S)<e$ inside $\Omega_{\text{tur}}$.
\end{enumerate}
By definition, we have $\bar{v}\in X_0$.
For any $v\in X_0$ and $t\in (0,T]$, since
$$
\frac{1}{2}\int_{\R^3}|v(t,x)|^2\dif x\leq E(t),
$$
we have
$$
V(t)=\frac{v(t)}{\max\{1,|\log t|\}^{\nicefrac{1}{2}}}\in B,
$$
for some closed ball in $L^2$. By the Banach–Alaoglu theorem, we can take $d_B$ by the metric of $B$ with respect to the weak convergence. Then,
$$
d(v_1,v_2)
=\sup_{t\in(0,T]}d_B(V_1(t),V_2(t)),
$$
defines a metric on $C_{\log t}L^2_{\text{weak}}$. Let $X$ be the closure of $X_0$ with respect to the metric $d$. Then, $X$ is a complete metric subspace of $C_{\log t}L^2_{\text{weak}}$. 

\begin{lemma}\label{lemma:X}
Any $v\in X$ satisfies the conditions \ref{sub:i}, \ref{sub:ii} and \ref{sub:iii} with 
\begin{equation}\label{eq:esub<e}
e_{\text{sub}}(v,S)\leq e,
\end{equation}
almost everywhere in $\Omega_{\text{tur}}$.
\end{lemma}
\begin{proof}
Since $\bar{X}_0=X$, there exists a sequence $(v_j)\subset X_0$ converging to $v$.
Let $S_j$ be the corresponding traceless symmetric tensor from \ref{sub:i}. By applying the conditions \ref{sub:ii}, \ref{sub:iii} and $e\sim 1/t$, we deduce that the sequence $(tS_j)$ is bounded in $L^\infty([0,T]\times\R^2)$. By the Banach-Alaoglu theorem, we may assume (by taking a subsequence if necessary) that $t S_j$ converges to some $tS$ in $L_{\text{weak}^*}^\infty([0,T]\times\R^2)$. By applying the Mazur lemma, we deduce
\eqref{eq:esub<e}. Finally, it follows that $(v,\bar{\pi},S)$ is a weak solution to the Euler-Reynolds equation.
\end{proof}

Given a (space-time) cylinder $I\times\Omega\subset\subset\Omega_{\text{tur}}$, we consider the functional
$J:X\to[0,\infty)$
$$
J(v)=\sup_{t\in I}\int_\Omega\left(e-\frac{1}{2}|v|^2\right)(t,x)\dif x.
$$

\begin{lemma}\label{lemma:J}
The functional $J$ is well defined.
Moreover, 
$J(v)=0$ if and only if 
$S=v\ocircle v$ in $I\times\Omega$. 
\end{lemma}
\begin{proof}
By \eqref{eq:esub<e}, any $v\in X\subset C_{\log t}L_{\text{weak}}^2$ satisfies 
$$\frac{1}{2}|v(t,x)|^2\leq e(t,x),$$
for almost every $x\in\Omega$ and all $t\in I$.
Moreover,
since $e$ is continuous in $I\times\Omega\subset\subset\Omega_{\text{tur}}$, we have
$$
J(v)\leq
\sup_{t\in I}\int_\Omega e(t,x)\dif x
<\infty.
$$
Suppose now that $J(v)=0$ and let $S$ be the traceless symmetric tensor from Lemma \ref{lemma:X}. By \eqref{eq:esub<e}, we have
$$
\frac{1}{2}|v|^2
+\frac{3}{2}\lambda_{\max}(S-v\ocircle v)\leq e = \frac{1}{2}|v|^2,
$$
and thus $S=v\ocircle v$ in $I\times\Omega$.
\end{proof}

Next, we check that any continuity point of the functional satisfies $J(v)=0$, which turns out to form a residual set in $X$.
The proofs appear in \cite[Lemma 4.3 and Proposition 4.5]{DeLellisSzekelyhidi10}.
Since the time interval $I$ is away from $t=0$, the same proofs hold in our space $X$.

\begin{lemma}\label{lemma:residual}
The functional $J$ is upper-semicontinuous. Hence,
$$
X_J=\{v\in X\,:\, J\text{ is continuous at }v\}
$$
is a countable intersection of open and dense sets.
\end{lemma}

\begin{prop}[Perturbation property]
For all $\alpha>0$ there exists $\beta(\alpha,I\times\Omega)>0$ such that, whenever $v\in X_0$ satisfies
$$
J(v)\geq\alpha,
$$
there exists a sequence $(v_k)\subset X_0$ converging to $v$ and satisfying
$$
\limsup_{k\to\infty}J(v_k)\leq J(v) - \beta.
$$
\end{prop}

\begin{cor}\label{cor:XJ0}
It holds
$$
X_J\subset J^{-1}(0).
$$
\end{cor}
\begin{proof}
We prove it by contradiction. Let us suppose that there exists $v\in X_J$ with $J(v)>0$, and take $\alpha=2J(v)$. Since $X=\bar{X}_0$, there exists a sequence $(v_j)\subset X_0$ converging to $v$. Since $J$ is continuous at $v$, we may a assume (by taking a subsequence if necessary) that $J(v_j)\geq\frac{1}{2}J(v)=\alpha>0$ for all indices $j$. By the perturbation property, for each $v_j$ there exists a sequence $(v_{j,k})\subset X_0$ converging to $v_j$ and satisfying
$$
\limsup_{k\to\infty}J(v_{j,k})\leq J(v_j) - \beta(\alpha,I\times\Omega).
$$
Since $X$ is a complete metric space, we can construct a diagonal sequence $v_{j,k(k)}\to v$ satisfying $J(v_{j,k(k)})\nrightarrow J(v)$, which contradicts $v\in X_J$.
\end{proof}

\subsection{Proof of Theorem \ref{thm:main}}

Firstly, we consider a covering $\{I_j\times\Omega_j\}$ of $\Omega_{\text{tur}}$, and the corresponding functionals $\{J_j\}$. By Lemma \ref{lemma:residual} and Corollary \ref{cor:XJ0}, the set
$$
X_{\text{sol}}=\bigcap_{j}J_j^{-1}(0)
$$
contains a countable intersection of open and dense sets in $X$. Since $X$ is a complete metric space, the Baire category theorem implies that $X_{\text{sol}}$ is dense. 

Let $v\in X_{\text{sol}}$. 
On the one hand, recall that the conditions \ref{sub:i} and \ref{sub:ii} are already satisfied by Lemma \ref{lemma:X}. In particular, $v=\bar{v}$ is analytic outside $\Omega_{\text{tur}}$.
On the other hand, we have 
$$
\frac{1}{2}|v|^2=e.
$$
Hence, by \eqref{eq:trueenergy}, we have
$$
E_v(t)
=\frac{1}{2}\int_{\R^2}|v(t,x)|^2\dif x
=\int_{\R^2}e(t,x)\dif x
=E(t).
$$
Since $v\in C_{\log t}L_{\text{weak}}^2$ and $|v|^2=2e\in C_{\log t}L^1$, we also have $v\in C_{\log t}L^2$.
Moreover, by Lemma \ref{lemma:J}, we have $S=v\ocircle v$ in $\Omega_{\text{tur}}$. Hence, $v$ is a weak solution to the Euler equation with pressure
$$
p=\bar{\pi}-\frac{2}{3}e.
$$
We remark that the pressure is also analytic outside $\Omega_{\text{tur}}$. Moreover, since the voriticty is continuous by Lemma \ref{lem:wprofileintegralconditions}, it is not difficult to check that the pressure is indeed continuous in $\R^3$ for positive times. 

Finally, we check that $v\in C([0,T],L^{2^-})$. Let us fix $1<p<2$.
On the one hand, the subsolution already satisfies $\bar{v}\in C([0,T],L^{p})$. 
Since $v=\bar{v}$ outside $\Omega_{\mathrm{tur}}\subset [0,T]\times B_R$ for some $R>0$, we also have $v\in C([0,T],L^{p}(\R^3\setminus B_R))$.
On the other hand, $v\in C([t_0,T],L^2(B_R))\subset C([t_0,T],L^p(B_R))$ for any $t_0>0$. Therefore, it remains to check the continuity at $t=0$. Since
$$
\|(v-\bar{v})(t)\|_{L^p}^p
=\int_{\Omega_{\text{tur}}(t)}
|(v-\bar{v})(t,x)|^p\dif x
\lesssim
c(t)^{2-p},
$$
we conclude that
$$
\|v(t)-v^\circ\|_{L^p}
\leq
\|(v-\bar{v})(t)\|_{L^p}
+\|\bar{v}(t)-v^\circ\|_{L^p}
\to 0,
$$
as $t\to 0$.

\appendix

\section{Axisymmetric without swirl Biot-Savart law}\label{sec:ABS}

In this section we recall the proof of Lemma \ref{lemma:BSlaw}.
Firstly, the Biot-Savart law in $\R^3$ is given by
$$
\bar{v}(t,x)
=\int_{\R^3}K_{3}(x-x')\bar{\omega}(t,x')\dif x'
=-\frac{1}{4\pi}\int_{\R^3}\frac{(x-x')\times\bar{\omega}(t,x')}{|x-x'|^3}\dif x'.
$$
By writing
$$x-x'=(re^{i\theta}-r'e^{i\theta'},z-z')
,$$
and
$$\bar{\omega}(t,x)
=\bar{\omega}_\theta(t,r,z)e_\theta
,$$
we deduce that
\begin{align*}
(x-x')\times e_{\theta'}
&=
\left|\begin{array}{ccc}
e_1 & e_2 & e_3 \\
r\cos\theta-r'\cos\theta'
& r\sin\theta-r'\sin\theta'
& z-z' \\
-\sin\theta' & \cos\theta' & 0
\end{array}\right|\\
&=-(z-z')e_{r'}
+(r\cos(\theta-\theta')-r')e_z,
\end{align*}
and also ($\varphi=\theta'-\theta$)
$$
|x-x'|^2
=|re^{i(\theta-\theta')}-r'|^2+|z-z'|^2
=rr'(2(1-\cos\varphi)+s),
$$
where
$$\bar{\zeta}=\frac{\zeta-\zeta'}{\sqrt{rr'}},
\quad\quad
s=|\bar{\zeta}|^2.$$
On the one hand, we have
\begin{align*}
&\int_0^{2\pi}
\int_{\mathbb{H}}\bar{\omega}_\theta\frac{(z-z')}{|x-x'|^3}e^{i\theta'} r'\dif (r',z')\dif\theta'\\
&=e^{i\theta}
\int_{\mathbb{H}}
\frac{(z-z')r'}{(rr')^{\nicefrac{3}{2}}}\bar{\omega}_\theta
\int_0^{2\pi}
\frac{e^{i\varphi}}{(2(1-\cos\varphi)+s)^{3/2}}\dif\varphi\dif (r',z').
\end{align*}
Thus,
\begin{align*}
\bar{v}_r
=\frac{1}{2\pi r}\int_{\mathbb{H}}
\frac{z-z'}{\sqrt{rr'}}
G\bar{\omega}_\theta\dif\zeta',
\end{align*}
where
\begin{equation}\label{eq:GBiotSavart}
G(s)
=\int_0^\pi
\frac{\cos\varphi}{(2(1-\cos\varphi)+s)^{\nicefrac{3}{2}}}\dif\varphi.
\end{equation}
On the other hand, we have
\begin{align*}
\int_{\mathbb{H}}&
\frac{r'}{(rr')^{\nicefrac{3}{2}}}
\bar{\omega}_\theta
\int_0^{2\pi}
\frac{r\cos\varphi-r'}{(2(1-\cos\varphi)+s)^{\nicefrac{3}{2}}}
\dif\varphi
\dif(r',z')\\
=&\int_{\mathbb{H}}
\frac{(r-r')r'}{(rr')^{\nicefrac{3}{2}}}
\bar{\omega}_\theta
\int_0^{2\pi}
\frac{\cos\varphi}{(2(1-\cos\varphi)+s)^{\nicefrac{3}{2}}}
\dif\varphi
\dif(r',z')\\
&-\int_{\mathbb{H}}
\frac{(r')^2}{(rr')^{\nicefrac{3}{2}}}
\bar{\omega}_\theta
\int_0^{2\pi}
\frac{1-\cos\varphi}{(2(1-\cos\varphi)+s)^{\nicefrac{3}{2}}}
\dif\varphi
\dif(r',z'),
\end{align*}
and thus
$$
\bar{v}_z
=-\frac{1}{2\pi r}\int_{\mathbb{H}}
\frac{r-r'}{\sqrt{rr'}}
G\bar{\omega}_\theta\dif\zeta'
+\frac{1}{2\pi r}\int_{\mathbb{H}}
\sqrt{\frac{r'}{r}}
H\bar{\omega}_\theta\dif\zeta',
$$
where
\begin{equation}\label{eq:HBiotSavart}
H(s)
=\int_0^\pi
\frac{1-\cos\varphi}{(2(1-\cos\varphi)+s)^{\nicefrac{3}{2}}}\dif\varphi.
\end{equation}
Therefore, we can write compactly
$$
\bar{v}=
-\frac{i}{2\pi r}
\int_{\mathbb{H}}
\bar{\zeta} G\bar{\omega}_\theta\dif\zeta'
+\frac{i}{2\pi r}
\int_{\mathbb{H}}
\sqrt{\frac{r'}{r}}
H\bar{\omega}_\theta\dif\zeta'.$$
It can be checked using elliptic integrals that this formula agrees with the one in \cite{FengSverak15}.

\begin{lemma}\label{lemma:Gestimates}
The function $G(s)$ defined in \eqref{eq:GBiotSavart} satisfies 
$$G(s)=\frac{1}{s}+O(\log s),$$
as $s\to 0$.
\end{lemma}
\begin{proof}
We make a Taylor expansion of the cosine in the numerator to obtain
$$
\begin{aligned}
G(s)=&\int_0^\pi
\frac{1}{(2(1-\cos\varphi)+s)^{\nicefrac{3}{2}}}\dif\varphi
-\frac{1}{2}\int_0^\pi
\frac{\varphi^2}{(2(1-\cos\varphi)+s)^{\nicefrac{3}{2}}}\dif\varphi\\
&+\int_0^\pi
\frac{O(\varphi^4)}{(2(1-\cos\varphi)+s)^{\nicefrac{3}{2}}}\dif\varphi
= G_1+G_2+G_3.
\end{aligned}$$
We further split $G_1$ in the following way

$$
\begin{aligned}
G_1(s)&=\int_0^\pi
\frac{1}{(\varphi^2+s)^{\nicefrac{3}{2}}}\dif\varphi+\int_0^\pi
\frac{1}{(2(1-\cos\varphi)+s)^{\nicefrac{3}{2}}}-\frac{1}{(\varphi^2+s)^{\nicefrac{3}{2}}}\dif\varphi.
\end{aligned}$$
The first integral can be computed exactly and is equal to 
$$\int_0^\pi
\frac{1}{(\varphi^2+s)^{\nicefrac{3}{2}}}\dif\varphi=\frac{\pi}{s\sqrt{s+\pi^2}}
=\frac{1}{s}+O(1).$$
For the second integral is equal to
$$
\int_0^\pi
\frac{(\varphi^2+s)^3-(2(1-\cos(\varphi))+s)^3}{(2(1-\cos\varphi)+s)^{\nicefrac{3}{2}}(\varphi^2+s)^{\nicefrac{3}{2}}((2(1-\cos\varphi)+s)^{\nicefrac{3}{2}}+(\varphi^2+s)^{\nicefrac{3}{2}})}\dif\varphi.
$$
Note that there exist real constants $C_1,C_2>0$ such that $C_1\varphi^2\leq 1-\cos(\varphi)\leq C_2\varphi^2$ for $0\leq \varphi \leq \pi/2$. For $\pi/2\leq \varphi\leq \pi$, the integrand is bounded. Therefore, in the last integral we can cancel the $s^3$ terms and bound the absolute value of the integral by 
$$
\int_0^{\nicefrac{\pi}{2}}\frac{O(\varphi^4)O((\varphi^2+s)^2)}{(\varphi^2+s)^{\nicefrac{9}{2}}}d\varphi+O(1)\lesssim 
\int_0^{\nicefrac{\pi}{2}}\frac{d\varphi}{(\varphi^2+s)^{\nicefrac{1}{2}}}+O(1)=O(\log s).
$$
The same idea let us bound 
$$|G_2|\leq O(\log s), \quad |G_3|\leq O(1).$$
\end{proof}

\begin{lemma}\label{lemma:Hestimates}
The function $H(s)$ defined in \eqref{eq:HBiotSavart} satisfies 
$$H(s)=-\frac{1}{4}\log s+O(1),$$
as $s\to 0$.
\end{lemma}
\begin{proof}
We follow the same strategy than in Lemma \ref{lemma:Gestimates}. We can split
$$H(s)=\frac{1}{2}\int_0^\pi\frac{\varphi^2}{(2(1-\cos \varphi)+s)^{\nicefrac{3}{2}}}d\varphi+\int_0^\pi\frac{O(\varphi^4)}{(2(1-\cos \varphi)+s)^{\nicefrac{3}{2}}}d\varphi.$$
We split the first integral into
$$\int_0^\pi\frac{\varphi^2}{(2(1-\cos \varphi)+s)^{\nicefrac{3}{2}}}d\varphi
=\int_0^\pi\frac{\varphi^2}{(\varphi^2+s)^{\nicefrac{3}{2}}}d\varphi
+\int_0^\pi\frac{\varphi^2}{(2(1-\cos \varphi)+s)^{\nicefrac{3}{2}}}-\frac{\varphi^2}{(\varphi^2+s)^{\nicefrac{3}{2}}}d\varphi.$$
The first part can be computed as
$$\int_0^\pi\frac{\varphi^2}{(\varphi^2+s)^{\nicefrac{3}{2}}}d\varphi=\arcsinh\big(\frac{\pi}{\sqrt{s}}\big)-\frac{\pi}{\sqrt{s+\pi^2}}=-\frac{1}{2}\log s+O(1).$$
The remainder terms are bounded similarly.
\end{proof}

\section{Antidivergence operator}\label{section:antidivergence}
In this section, we construct and give estimates for the antidivergence operator. We follow the construction of \cite[Section 10]{IsettOh16}. 

\begin{lemma}\label{lem:antidivergence}
Let $c<1$ and $f$ be a smooth vector field supported on a ball $B(x_0,c)$ such that
\begin{equation}\label{conditionsantidivergence}
\int f(x)\dif x=0, \quad \int f(x)\cdot (x-x_0)^\perp \dif x=0.
\end{equation}
Then, there exists a symmetric tensor $\mathcal{R}f$ solving 
\begin{equation}\label{divergenceequation}
\mathrm{div}(\mathcal{R}f)=f,
\end{equation}
such that $\mathcal{R}f$ is supported on $B(x_0,c)$ and 
$$\|\mathcal{R}f\|_{L^\infty}\lesssim c\|f\|_{L^\infty}.$$
\end{lemma}
\begin{proof}
We assume $x_0=0$ for simplicity. Using Einstein convention for summation, we rewrite
$$
\begin{aligned}
\hat{f}^l(\xi)&=\hat{f}^l(0)+\left(\int_0^1 \partial_k \hat{f}^l(\sigma\xi)\dif \sigma\right)\xi_k\\
&=\hat{f}^l(0)+\frac{1}{2}\left(\int_0^1 (\partial_k \hat{f}^l-\partial_l \hat{f}^k)(\sigma\xi)\dif \sigma\right)\xi_k+\frac{1}{2}\left(\int_0^1 (\partial_k \hat{f}^l+\partial_l \hat{f}^k)(\sigma\xi)\dif \sigma\right)\xi_k.
\end{aligned}
$$
We integrate by parts to further rewrite
$$
\begin{aligned}
\frac{1}{2}\left(\int_0^1 (\partial_k \hat{f}^l-\partial_l \hat{f}^k)(\sigma\xi)\dif \sigma\right)\xi_k&=-\frac{1}{2}\left(\int_0^1 \partial_\sigma (1-\sigma)(\partial_k \hat{f}^l-\partial_l \hat{f}^k)(\sigma\xi)\dif \sigma\right)\zeta_k \\
&\hspace{-4.0cm}=-\frac{1}{2} (1-\sigma)(\partial_k \hat{f}^l-\partial_l \hat{f}^k)(\sigma\zeta)\zeta_k\big|_{\sigma=0}^{\sigma=1}+\frac{1}{2}\left(\int_0^1 (1-\sigma)(\partial_j\partial_k \hat{f}^l-\partial_j\partial_l \hat{f}^k)(\sigma\zeta)\dif \sigma\right)\zeta_j\zeta_k\\
&\hspace{-4.0cm}=\frac{1}{2}(\partial_k \hat{f}^l-\partial_l \hat{f}^k)(0)\zeta_k+\frac{1}{2}\left(\int_0^1 (1-\sigma)(\partial_j\partial_k \hat{f}^l-\partial_j\partial_l \hat{f}^k)(\sigma\zeta)\dif \sigma\right)\zeta_j\zeta_k\\
&\hspace{-4.0cm}=\frac{1}{2}(\partial_k \hat{f}^l-\partial_l \hat{f}^k)(0)\zeta_k
+\frac{1}{2}\left(\int_0^1 (1-\sigma)\zeta_k(\partial_j\partial_k \hat{f}^l-\partial_k\partial_l \hat{f}^j)(\sigma\zeta)\dif \sigma\right)\zeta_j\\
&-\left(\int_0^1 (1-\sigma)\zeta_k(\partial_j\partial_l \hat{f}^k)(\sigma\zeta)\dif \sigma\right)\zeta_j.
\end{aligned}
$$
Using \eqref{conditionsantidivergence}, we obtain
$$
\begin{aligned}
\hat{f}^l(\zeta)=&\frac{1}{2}\left(\int_0^1 (\partial_j \hat{f}^l+\partial_l \hat{f}^j)(\sigma\zeta)\dif \sigma\right)\zeta_j\\
&+\frac{1}{2}\left(\int_0^1 (1-\sigma)\zeta_k(\partial_j\partial_k \hat{f}^l-\partial_k\partial_l \hat{f}^j)(\sigma\zeta)\dif \sigma\right)\zeta_j\\
&-\left(\int_0^1 (1-\sigma)\zeta_k(\partial_j\partial_l \hat{f}^k)(\sigma\zeta)\dif \sigma\right)\zeta_j.
\end{aligned}$$
Note that each summand is symmetric in $l,j$. Thus, inverting the Fourier transform the terms in brackets, we obtain a symmetric matrix whose divergence is $f$:

$$
\begin{aligned}
r^{lj}&=r_0^{lj}+r_1^{lj}+r_2^{lj},\\
r_0^{lj}&=-\frac{1}{2}\int_0^1\left(\frac{x_j}{\sigma}f^{l}(\frac{x}{\sigma})+\frac{x_l}{\sigma}f^{j}(\frac{x}{\sigma})\right)\frac{\dif \sigma}{\sigma^2},\\
r_1^{lj}&=\frac{1}{2}\frac{\partial}{\partial x_k}\int_0^1(1-\sigma)\left(\frac{x_jx_k}{\sigma^2}f^{l}(\frac{x}{\sigma})+\frac{x_lx_k}{\sigma^2}f^{j}(\frac{x}{\sigma})\right)\frac{\dif \sigma}{\sigma^2},\\
r_2^{lj}&=-\frac{\partial}{\partial x_k}\int_0^1(1-\sigma)\left(\frac{x_lx_j}{\sigma^2}f^{k}(\frac{x}{\sigma})\right)\frac{\dif \sigma}{\sigma^2}.
\end{aligned}
$$
Since $0\leq\sigma\leq 1,$ it is direct that $r^{lj}$ is supported on the same ball than $f$. The above expressions define a distributional solution to \eqref{divergenceequation}, but there is a singularity at $x=0$. One can move the singularity to a point $y$ conjugating $r^{lj}$ by translation by $y$, obtaining 
$$
\begin{aligned}
^{(y)}r^{lj}&=^{(y)}r_0^{lj}+^{(y)}r_1^{lj}+^{(y)}r_2^{lj},\\
^{(y)}r_0^{lj}&=-\frac{1}{2}\int_0^1\left(\frac{(x-y)_j}{\sigma}f^{l}(\frac{x-y}{\sigma}+y)+\frac{(x-y)_l}{\sigma}f^{j}(\frac{x-y}{\sigma}+y)\right)\frac{\dif \sigma}{\sigma^2},\\
^{(y)}r_1^{lj}&=\frac{1}{2}\frac{\partial}{\partial x_k}\int_0^1(1-\sigma)\left(\frac{(x-y)_j(x-y)_k}{\sigma^2}f^{l}(\frac{x-y}{\sigma}+y)+\frac{(x-y)_l(x-y)_k}{\sigma^2}f^{j}(\frac{x-y}{\sigma}+y)\right)\frac{\dif \sigma}{\sigma^2},\\
^{(y)}r_2^{lj}&=-\frac{\partial}{\partial x_k}\int_0^1(1-\sigma)\left(\frac{(x-y)_l(x-y)_j}{\sigma^2}f^{k}(\frac{x-y}{\sigma}+y)\right)\frac{\dif \sigma}{\sigma^2}.
\end{aligned}
$$
Note that if $y\in B(0,c)$, then $\mathrm{supp}\,r^{lj}\subset B(0,c)$ again. Take a mollifier $\phi(y)$ supported on $B(0,c)$. To remove the singularity in $r^{lj}$, we mollify using $\phi$, obtaining finally the solution
$$R^{lj}(x)=\int \prescript{(y)}{}r^{lj}(x)\phi(y)\dif y.$$
Now we seek a more convenient expression for $R^{lj}$. Denote $z=x-y$. Using that 
$$
\partial_{y_k}(z)=-\partial_{x_k}(z), \quad \partial_{y_k}\big(f^{*}(\frac{z}{\sigma}+y)\big)=-(1-\sigma)\partial_{x_k}\big(f^{*}(\frac{z}{\sigma}+y)\big),
$$
we can rewrite
$$
\begin{aligned}
R_2^{lj}(x)=&-\int\phi(y)\int_0^1(1-\sigma)\partial_{x_k}\big(\frac{z_lz_j}{\sigma^2}\big)f^k(\frac{z}{\sigma}+y)\frac{\dif \sigma}{\sigma^2}\dif y
+\int\phi(y)\int_0^1\frac{z_lz_j}{\sigma^2}\partial_{y_k}\big(f^k(\frac{z}{\sigma}+y)\big)\frac{\dif \sigma}{\sigma^2}\dif y\\
=&-\int\phi(y)\int_0^1(1-\sigma)\partial_{x_k}\big(\frac{z_lz_j}{\sigma^2}\big)f^k(\frac{z}{\sigma}+y)\frac{\dif \sigma}{\sigma^2}\dif y
+\int\phi(y)\partial_{y_k}\int_0^1\frac{z_lz_j}{\sigma^2}\big(f^k(\frac{z}{\sigma}+y)\big)\frac{\dif \sigma}{\sigma^2}\dif y\\
&-\int\phi(y)\int_0^1\partial_{y_k}\big(\frac{z_lz_j}{\sigma^2}\big)\big(f^k(\frac{z}{\sigma}+y)\big)\frac{\dif \sigma}{\sigma^2}\dif y\\
=&\int\phi(y)\int_0^1\partial_{x_k}\big(\frac{z_lz_j}{\sigma}\big)f^k(\frac{z}{\sigma}+y)\frac{\dif \sigma}{\sigma^2}\dif y
-\int\partial_{y_k}\phi(y)\int_0^1\frac{z_lz_j}{\sigma^2}\big(f^k(\frac{z}{\sigma}+y)\big)\frac{\dif \sigma}{\sigma^2}\dif y.
\end{aligned}
$$
Similarly, we can obtain
$$
\begin{aligned}
R_1^{lj}(x)=&-\frac{1}{2}\int\phi(y)\int_0^1\left(\partial_{x_k}\big(\frac{z_jz_k}{\sigma}\big)f^l(\frac{z}{\sigma}+y)+\partial_{x_k}\big(\frac{z_lz_k}{\sigma}\big)f^j(\frac{z}{\sigma}+y)\right)\frac{\dif \sigma}{\sigma^2}\dif y\\
&+\frac{1}{2}\int\partial_{y_k}\phi(y)\int_0^1\left(\frac{z_jz_k}{\sigma^2}f^{l}(\frac{z}{\sigma}+y)+\frac{z_lz_k}{\sigma^2}f^{j}(\frac{z}{\sigma}+y)\right)\frac{\dif \sigma}{\sigma^2}\dif y.
\end{aligned}
$$
Taking into account that 
$$
\partial_{x_k}(z_lz_j)f^k-\frac{1}{2}\big(\partial_{x_k}(z_jz_k)f^l+\partial_{x_k}(z_lz_k)f^j\big)=-\frac{1}{2}(z_jf^l+z_lf^j),
$$
we get the alternative expression
$$
\begin{aligned}
R^{lj}(x)=&-\int\phi(y)\int_0^1\left(\frac{z_j}{\sigma}f^{l}(\frac{z}{\sigma}+y)+\frac{z_l}{\sigma}f^{j}(\frac{z}{\sigma}+y)\right)\frac{\dif \sigma}{\sigma^2}\dif y\\
&+\frac{1}{2}\int\partial_{y_k}\phi(y)\int_0^1\left(\frac{z_jz_k}{\sigma^2}f^{l}(\frac{z}{\sigma}+y)+\frac{z_lz_k}{\sigma^2}f^{j}(\frac{z}{\sigma}+y)\right)\frac{\dif \sigma}{\sigma^2}\dif y\\
&-\int\partial_{y_k}\phi(y)\int_0^1\frac{z_lz_j}{\sigma^2}\big(f^k(\frac{z}{\sigma}+y)\big)\frac{\dif \sigma}{\sigma^2}\dif y=\widetilde{R}_0^{lj}+\widetilde{R}_1^{lj}+\widetilde{R}_2^{lj}.
\end{aligned}
$$
We can estimate for $x\in B(0,c)$
$$
\begin{aligned}
|\widetilde{R}_0^{lj}(x)|&\lesssim\int_{B(0,c)}\int_{\frac{|z|}{c}}^2\frac{|z|}{\sigma}\|\phi\|_{L^\infty}\|f\|_{L^\infty}\frac{\dif \sigma}{\sigma^2}\dif y
\lesssim \|\phi\|_{L^\infty}\|f\|_{L^\infty}\int_{B(0,c)}\big(4-\frac{c^2}{|z|^2}\big)|z|\dif y\\
&\lesssim \|\phi\|_{L^\infty}\|f\|_{L^\infty}\int_{B(0,c)}\big(4+\frac{c^2}{|z|^2}\big)|z|\dif y\lesssim c^3\|\phi\|_{L^\infty}\|f\|_{L^\infty}.
\end{aligned}
$$
Similarly, we can bound $|\widetilde{R}_1^{lj}+\widetilde{R}_2^{lj}|\lesssim c^4\|\nabla\phi\|_{L^\infty}\|f\|_{L^\infty}$. Choosing $\phi$ so that $$\|D^j\phi\|_{L^\infty}\lesssim c^{-(2+j)},$$ 
we conclude the proof by taking $\mathcal{R}f=R$.
\end{proof}

\subsection*{Acknowledgments}
The authors are grateful to the Applied Analysis Research Group at the Max Planck Institute for Mathematics in the Sciences in Leipzig, where part of this work was done, for their warm hospitality and stimulating discussions.
The authors were partially supported by the MICINN through the grants EUR2020-112271 and PID2020-114703GB-I00.
FG and AHT were partially supported by the Fundaci\'on de Investigaci\'on de la Universidad de Sevilla through the
grant FIUS23/0207, by the MICINN (Spain) through the grant PID2022-140494N, and by the Junta de Andaluc\'ia through the grant P20-00566. FG was partially supported by MINECO grant RED2018-102650-T (Spain). FG acknowledges support
from IMAG, funded by MICINN through the Maria de Maeztu Excellence Grant CEX2020-001105-M/AEI/10.13039/501100011033. AHT was partially supported by the grant FPU19/02748,
funded by the Spanish Ministry of Universities.
FM was partially supported by the 
Max Planck Institute for Mathematics in the Sciences, the Spanish Ministry of Science and Innovation through the Severo Ochoa Programme for Centres of Excellence in R\&D
(CEX2019-000904-S) and the ERC Advanced Grant 834728.

\bibliographystyle{abbrv}
\bibliography{Vortexfilament_convexintegration}

\begin{flushleft}
	Francisco Gancedo\\
	\textsc{Departamento de An\'alisis Matem\'atico \& IMUS, Universidad de Sevilla\\
	41012 Sevilla, Spain}\\
	\textit{E-mail address:
    fgancedo@us.es}
\end{flushleft}

\begin{flushleft}
	Antonio Hidalgo-Torn\'e\\
	\textsc{Departamento de An\'alisis Matem\'atico \& IMUS, Universidad de Sevilla\\
	41012 Sevilla, Spain}\\
	\textit{E-mail address:
    ahtorne@us.es}
\end{flushleft}

\begin{flushleft}
	Francisco Mengual\\
	\textsc{Max Planck Institute for Mathematics in the Sciences\\
		04103 Leipzig, Germany}\\
	\textit{E-mail address: fmengual@mis.mpg.de}
\end{flushleft}

\end{document}